\documentclass[10pt]{article}

\usepackage{palatino}
\usepackage{float}

\usepackage{url}
\usepackage{amsmath,amsthm,amssymb,amsbsy}
\usepackage{mathdots}
\usepackage{paralist}
\usepackage{xcolor}
\usepackage{color}
\usepackage{graphicx}
\usepackage{algorithm,algpseudocode}
\usepackage{comment}

\usepackage{fancyhdr}
\usepackage{cite}
\usepackage{cleveref}
\usepackage{enumerate}

\newtheorem{thm}{Theorem}%[section] %(If you want theorem numbered
\newtheorem{lemma}{Lemma}%[section] %%    with section number.
\newtheorem{cor}{Corollary}%[section]
\newtheorem{definition}{Definition}%[section]

\theoremstyle{remark}
\newtheorem{remark}{Remark}

\newcommand{\R}{\mathbb{R}}
\def \real    { \mathbb{R} }

\newcommand{\C}{\mathbb{C}}
\def \complex{\mathbb{C}}

\newcommand{\Z}{\mathbb{Z}}
\newcommand{\N}{\mathbb{N}}

\newcommand{\E}{\operatorname{E}}

\newcommand{\e}{\mathrm{e}}

\newcommand{\vct}[1]{\boldsymbol{#1}}
\newcommand{\mtx}[1]{\boldsymbol{#1}}
\newcommand{\Null}{\operatorname{Null}}
\newcommand{\Range}{\operatorname{Range}}

\newcommand{\rank}{\operatorname{rank}}

\newcommand{\dimension}{\operatorname{dim}}

\newcommand{\set}[1]{\mathcal{#1}}
\newcommand{\sinc}{\text{sinc}}

\newcommand{\ceil}[1]{\left\lceil #1 \right\rceil}
\newcommand{\eps}{\epsilon}

\newcommand{\calT}{\mathcal{T}}

\newcommand{\calB}{\mathcal{B}}

\newcommand{\va}{\vct{a}}
\newcommand{\vb}{\vct{b}}

\newcommand{\ve}{\vct{e}}

\newcommand{\vg}{\vct{g}}

\newcommand{\vs}{\vct{s}}

\newcommand{\vv}{\vct{v}}

\newcommand{\vx}{\vct{x}}
\newcommand{\vy}{\vct{y}}

\newcommand{\mA}{\mtx{A}}
\newcommand{\mB}{\mtx{B}}
\newcommand{\mC}{\mtx{C}}
\newcommand{\mD}{\mtx{D}}
\newcommand{\mE}{\mtx{E}}
\newcommand{\mF}{\mtx{F}}

\newcommand{\mH}{\mtx{H}}
\newcommand{\mJ}{\mtx{J}}

\newcommand{\mL}{\mtx{L}}
\newcommand{\mM}{\mtx{M}}

\newcommand{\mS}{\mtx{S}}
\newcommand{\mT}{\mtx{T}}
\newcommand{\mU}{\mtx{U}}
\newcommand{\mV}{\mtx{V}}

\newcommand{\mX}{\mtx{X}}

\newcommand{\mZ}{\mtx{Z}}

\newcommand{\mLambda}{\mtx{\Lambda}}

\newcommand{\mId}{{\bf I}}

\newcommand{\mzero}{{\bf 0}}

\newcommand{\setI}{\set{I}}

\newcommand{\setS}{\set{S}}

\graphicspath{{./figs/}}

\newlength{\imgwidth}
\newboolean{twoColVersion}
\setboolean{twoColVersion}{false}
\newcommand{\twoCol}[2]{\ifthenelse{\boolean{twoColVersion}} {#1} {#2} }

\newcommand\blfootnote[1]{
  \begingroup
  \renewcommand\thefootnote{}\footnote{#1}
  \addtocounter{footnote}{-1}
  \endgroup
}

\begin{document}

\title{The Fast Slepian Transform}

\vspace{2mm}
\author{Santhosh Karnik, Zhihui Zhu, Michael B. Wakin, Justin Romberg, Mark A. Davenport}

\maketitle

\begin{abstract}
The discrete prolate spheroidal sequences (DPSS's) provide an efficient representation for discrete signals that are perfectly timelimited and nearly bandlimited. Due to the high computational complexity of projecting onto the DPSS basis -- also known as the {\em Slepian basis} -- this representation is often overlooked in favor of the fast Fourier transform (FFT). We show that there exist fast constructions for computing approximate projections onto the leading Slepian basis elements. The complexity of the resulting algorithms is comparable to the FFT, and scales favorably as the quality of the desired approximation is increased. In the process of bounding the complexity of these algorithms, we also establish new nonasymptotic results on the eigenvalue distribution of discrete time-frequency localization operators. We then demonstrate how these algorithms allow us to efficiently compute the solution to certain least-squares problems that arise in signal processing. We also provide simulations comparing these fast, approximate Slepian methods to exact Slepian methods as well as the traditional FFT based methods.
%briefly survey applications of these constructions in XXX, YYY, ZZZ.
\end{abstract}

\blfootnote{S. Karnik, J. Romberg, and M. A. Davenport are with the School of Electrical and Computer Engineering, Georgia Institute of Technology, Atlanta, GA, 30332 USA (e-mail: skarnik1337@gatech.edu, jrom@ece.gatech.edu, mdav@gatech.edu). Z. Zhu and M. B. Wakin are with the  Department of Electrical Engineering, Colorado School of Mines, Golden, CO USA (e-mail: zzhu@mines.edu, mwakin@mines.edu). This work was supported by NSF grants CCF-1409261 and CCF-1409406.  A preliminary version of this paper highlighting some of the key results also appeared in~\cite{KarnikZWRD_Fast}.}

\section{Introduction}
\label{sec:intro}
Any bandlimited signal must have infinite duration.  No signal which is compactly supported in time can be bandlimited. These well-known mathematical facts stand in tension with the fact that real-world signals would seem to be {\em both} bandlimited {\em and} timelimited -- a real signal ought not to have energy at arbitrarily high frequencies and certainly ought to have a beginning and end.

One possible resolution of this ``paradox'' was provided by Landau, Pollak, and Slepian, who wrote a series of seminal papers exploring the degree to which a bandlimited signal can be {\em approximately} timelimited~\cite{SlepiP_ProlateI,LandaP_ProlateII,LandaP_ProlateIII,Slepi_ProlateIV,Slepi_ProlateV} (see also~\cite{Slepi_On,Slepi_Some} for beautiful and concise overviews of this body of work). In these papers, we find answers to questions such as ``Which bandlimited signals are most concentrated in time?'' and ``What is the (approximate) dimension of the space of signals which are both bandlimited {\em and} (approximately) timelimited?'' The answer to both of these questions turns out to involve a very special class of functions -- the {\em prolate spheroidal wave functions} (PSWF's) in the continuous case and the {\em discrete prolate spheroidal sequences} (DPSS's) in the discrete case. As shown by Landau, Pollak, and Slepian, these functions provide a natural basis to use in a wide variety of applications involving bandlimiting/timelimiting.

While this body of work has provided a great deal of theoretical insight into a range of problems, it has been somewhat less useful in terms of practical applications.  DPSS's, which form what we will refer to as the {\em Slepian basis}, provide the most natural basis to use in analyzing a finite-length vector of samples of a bandlimited signal. Nevertheless, they are rarely used in practice; the {\em discrete Fourier transform} (DFT) is a {\em far} more common choice. In many cases, this choice is being made not because the Fourier basis provides a more appropriate representation, but because the {\em fast Fourier transform} (FFT) gives us a highly-efficient method for working with the Fourier basis.  The Slepian basis, in contrast, comes with no such tools. Indeed, merely computing the DPSS's themselves (which lack a closed form solution) is a nontrivial computational challenge.  It is our purpose in this paper to fill this gap by providing computational tools comparable to the FFT for working with the Slepian basis. In the process we will also provide new nonasymptotic results concerning fundamental properties of DPSS's.

The key insight to these fast computational tools is observing the structural similarity between the prolate matrix and an orthogonal projection matrix corresponding to the span of the low frequency DFT vectors. The prolate matrix is a Toeplitz matrix whose entries are samples of the sinc function. The eigenvectors of this matrix are the DPSSs, and most of the eigenvalues are clustered very near zero or very near one. The orthogonal projection matrix corresponding to the span of the low frequency DFT vectors is a circulant matrix whose entries are samples of the digital sinc function (also known as the Dirchlet function). The eigenvectors of this matrix are the DFT vectors, and the eigenvalues are all either exactly zero or exactly one. These similarities motivate us to show that the difference between these two matrices is approximately a low rank matrix. From this, we get a bound on the number of eigenvalues of the prolate matrix which are not very close to either zero or one. This bound then allows us to approximate several matrices, which are related to the Slepian basis, as the sum of a Toeplitz matrix and a low rank matrix, thus giving rise to the fast computational tools.

\subsection{The Slepian basis}

To begin, we provide a formal definition of the Slepian basis and briefly describe some of the key results from Slepian's 1978 paper on DPSS's~\cite{Slepi_ProlateV}. Given any $N \in \N$ and $W \in (0,\tfrac{1}{2})$, the DPSS's are a collection of $N$ discrete-time sequences that are strictly bandlimited to the digital frequency range $|f| \le W$ yet highly concentrated in time to the index range $n = 0,1,\dots,N-1$. The DPSS's are defined to be the eigenvectors of a two-step procedure in which one first time-limits the sequence and then bandlimits the sequence. Before we can state a more formal definition, let us note that for a given discrete-time signal $x[n]$, we let
$$
X(f) = \sum_{n=-\infty}^\infty x[n] e^{-j 2 \pi f n}
$$
denote the {\em discrete-time Fourier transform} (DTFT) of $x[n]$. Next, we let $\calB_W$ denote an operator that takes a discrete-time signal, bandlimits its DTFT to the frequency range $|f| \le W$, and returns the corresponding signal in the time domain.
Additionally, we let $\calT_N$ denote an operator that takes an infinite-length discrete-time signal and zeros out all entries outside the index range $\{0,1,\dots,N-1\}$ (but still returns an infinite-length signal).
With these definitions, the DPSS's are defined in~\cite{Slepi_ProlateV} as follows.
\begin{definition}
Given any $N \in \N$ and $W \in (0,\tfrac{1}{2})$, the DPSS's are a collection of $N$ real-valued discrete-time sequences
$s_{N,W}^{(0)}, s_{N,W}^{(1)}, \dots, s_{N,W}^{(N-1)}$
that, along with the corresponding scalar eigenvalues $1 > \lambda_{N,W}^{(0)} > \lambda_{N,W}^{(1)} > \cdots > \lambda_{N,W}^{(N-1)} > 0,$
satisfy
\begin{equation}
\calB_W(\calT_N (s_{N,W}^{(\ell)})) = \lambda_{N,W}^{(\ell)} s_{N,W}^{(\ell)}
\label{eq:dpsseigmain}
\end{equation}
for all $\ell \in \{0,1,\dots,N-1\}$.
The DPSS's are normalized so that
\begin{equation}
\label{eq:dpssnorm}
\| \calT_N(s_{N,W}^{(\ell)}) \|_2 = 1
\end{equation}
for all $\ell \in \{0,1,\dots,N-1\}$.
\end{definition}

One of the central contributions of~\cite{Slepi_ProlateV} was to examine the behavior of the eigenvalues $\lambda_{N,W}^{(0)}, \dots, \lambda_{N,W}^{(N-1)}$. In particular,~\cite{Slepi_ProlateV} shows that the first $2NW$ eigenvalues tend to cluster extremely close to $1$, while the remaining eigenvalues tend to cluster similarly close to $0$. This is made more precise in the following lemma from~\cite{Slepi_ProlateV}.
\begin{lemma} \label{lem:SlepEig} Suppose that $W$ is fixed, and let $\rho \in (0,1)$ be fixed. Then there exist constants $C_0$ and $N_0$ (which may depend on $W$ and $\rho$) such that
\begin{equation}
\lambda_{N,W}^{(\ell)} \ge 1 - e^{- C_0 N} ~~ \mathrm{for~all}~ \ell \le 2NW(1-\rho) ~\mathrm{and~all}~N \ge N_0.
\label{eq:eig2nbminus}
\end{equation}
Similarly, for any fixed $\rho \in (0,\frac{1}{2W}-1)$ there exist constants $C_1$ and $N_1$ (which may depend on $W$ and $\rho$) such that
\begin{equation}
\lambda_{N,W}^{(\ell)} \le e^{- C_1 N} ~~ \mathrm{for~all}~ \ell \ge 2NW(1+\rho) ~\mathrm{and~all}~N \ge N_1.
\label{eq:eig2nbplus}
\end{equation}
\end{lemma}
This tells us that the range of the operator $\calB_W \calT_N$ has an effective dimension of $\approx 2NW$. Moreover, with only a few exceptions near the ``transition region'' at $\ell \approx 2NW$, we can reasonably approximate the eigenvalues $\lambda_{N,W}^{(\ell)}$ to be either 1 or 0. This will play a central role throughout our analysis.

\begin{comment}
Before proceeding, let us briefly mention several key properties of the DPSS's that will be useful in our subsequent analysis. First, it is clear from (\ref{eq:dpsseigmain}) that the DPSS's are, by definition, strictly bandlimited to the digital frequency range $|f| \le W$. Second, the DPSS's are also approximately time-limited to the index range $n = 0,1,\dots,N-1$.  Specifically, it can be shown that~\cite{Slepi_ProlateV}
\begin{equation}
\label{eq:snorm}
\| s_{N,W}^{(\ell)} \|_2 = \frac{1}{\sqrt{\lambda_{N,W}^{(\ell)}}}.
\end{equation}
Comparing~\eqref{eq:dpssnorm} with~\eqref{eq:snorm}, we see that for values of $\ell$ where $\lambda_{N,W}^{(\ell)} \approx 1$, nearly all of the energy in $s_{N,W}^{(\ell)}$ is contained in the interval $\{0,1,\dots,N-1\}$. Third, the DPSS's are orthogonal over $\Z$~\cite{Slepi_ProlateV}, i.e., for any $\ell,\ell' \in \{0,1,\dots,N-1\}$ with $\ell \neq \ell'$, $\langle s_{N,W}^{(\ell)}, s_{N,W}^{(\ell')} \rangle = 0$.
\end{comment}

Finally, we also note that while each DPSS actually has infinite support in time, several very useful properties hold for the collection of signals one obtains by time-limiting the DPSS's to the index range $n = 0,1,\dots,N-1$. First, it can be shown that~\cite{Slepi_ProlateV}
\begin{equation}
\| \calB_W( \calT_N (s_{N,W}^{(\ell)})) \|_2 = \sqrt{\lambda_{N,W}^{(\ell)}}.
\label{eq:apbl}
\end{equation}
Comparing~\eqref{eq:dpssnorm} with~\eqref{eq:apbl}, we see that for values of $\ell$ where $\lambda_{N,W}^{(\ell)} \approx 1$, nearly all of the energy in $\calT_N (s_{N,W}^{(\ell)})$ is contained in the frequencies $|f| \le W$.
%An illustration of four representative time-limited DPSS's and their DTFT's is provided in Figure~\ref{fig:dpssplots}.
%
While by construction the DTFT of any DPSS is perfectly bandlimited, the DTFT of the corresponding time-limited DPSS will only be concentrated in the bandwidth of interest for the first $\approx 2NW$ DPSS's.  As a result, we will frequently be primarily interested in roughly the first $2NW$ DPSS's.
%
\begin{comment}
\begin{figure*}
   \centering
   \begin{tabular*}{\linewidth}{@{\extracolsep{\fill}} cccc}
   \hspace*{-.18in} \includegraphics[width=.28\linewidth]{figs/dpss0_time} & \hspace*{-.37in}\includegraphics[width=.28\linewidth]{figs/dpss127_time} & \hspace*{-.37in} \includegraphics[width=.28\linewidth]{figs/dpss511_time} & \hspace*{-.37in}\includegraphics[width=.28\linewidth]{figs/dpss767_time} \\
   \hspace*{-.47in} \includegraphics[width=.28\linewidth]{figs/dpss0_freq} & \hspace*{-.38in} \includegraphics[width=.28\linewidth]{figs/dpss127_freq} & \hspace*{-.37in} \includegraphics[width=.28\linewidth]{figs/dpss511_freq} & \hspace*{-.37in}\includegraphics[width=.28\linewidth]{figs/dpss767_freq} \\
   \hspace*{-.1in} {\small $\ell=0$} & \hspace*{-.15in}{\small $\ell=127$} & \hspace*{-.14in}{\small  $\ell=511$} & \hspace*{-.17in}{\small $\ell=767$}
   \end{tabular*}
   \caption{\small \sl Illustration of four example DPSS's time-limited to the interval $[0,1,\dots,N-1]$ and the magnitude of their DTFT's.  In this example, $N=1024$ and $W = \frac{1}{4}$. Note that for $\ell$ up to approximately $2NW-1$ the energy of the spectrum is highly concentrated on the interval $[-W,W]$, and when $\ell$ is sufficiently larger than $2NW-1$ the energy of the spectrum is concentrated almost entirely outside the interval $[-W,W]$. \label{fig:dpssplots}}
\end{figure*}
\end{comment}
%
Second, the time-limited DPSS's are orthogonal~\cite{Slepi_ProlateV} so that for any $\ell,\ell' \in \{0,1,\dots,N-1\}$ with $\ell \neq \ell'$,
\begin{equation}
\label{eq:dpssorth}
 \langle \calT_N(s_{N,W}^{(\ell)}), \calT_N(s_{N,W}^{(\ell')}) \rangle  = 0.
\end{equation}
Finally, like the DPSS's, the time-limited DPSS's have a special eigenvalue relationship with the time-limiting and bandlimiting operators.
In particular, if we apply the operator $\calT_N$ to both sides of (\ref{eq:dpsseigmain}), we see that the sequences $\calT_N (s_{N,W}^{(\ell)})$ are actually eigenfunctions of the two-step procedure in which one first bandlimits a sequence and then time-limits the sequence.

These properties, together with the fact that our focus is primarily on providing computational tools for finite-length vectors, motivate our definition of the Slepian basis to be the restriction of the (time-limited) DPSS's to the index range $n = 0,1,\dots,N-1$ (discarding the zeros outside this range).
\begin{definition}
Given any $N \in \N$ and $W \in (0,\tfrac{1}{2})$, the Slepian basis is given by the vectors $\vs_{N,W}^{(0)}, \vs_{N,W}^{(1)}, \dots, \vs_{N,W}^{(N-1)} \in \real^N$ which are defined by restricting the time-limited DPSS's to the index range $n = 0,1,\dots,N-1$:
$$
\vs_{N,W}^{(\ell)}[n] := \calT_N(s_{N,W}^{(\ell)})[n] = s_{N,W}^{(\ell)}[n]
$$
for all $\ell,n \in \{0,1,\dots,N-1\}$.  For simplicity, we will often use the notation $\mS_{N,W}$ to denote the $N \times N$ matrix given by
\[
\mS_{N,W} = \begin{bmatrix} \vs_{N,W}^{(0)} & \cdots & \vs_{N,W}^{(N-1)}\end{bmatrix}.
\]
\end{definition}

Observe that combining~\eqref{eq:dpssnorm} and~\eqref{eq:dpssorth}, it follows that $\mS_{N,W}$ does indeed form an orthonormal basis for $\complex^N$ (or for $\real^N$).
However, following from our discussion above, the partial Slepian basis constructed using just the first $\approx 2NW $ basis elements will play a special role and can be shown to be remarkably effective for capturing the energy in a length-$N$ window of samples of a bandlimited signal (see~\cite{DavenportWakin2012CSDPSS} for further discussion). In such situations, we will also use the notation $\mS_K$ to denote the first $K$ columns of $\mS_{N,W}$ (where $N$ and $W$ are clear from the context and typically $K \approx 2NW$).

\subsection{The Slepian basis, the Fourier basis, and the prolate matrix}
\label{ssec:SlepFourProl}

In our discussion above we derived the Slepian basis by following the same approach as in~\cite{Slepi_ProlateV} and considering the time-limitations of the eigenfunctions of the operator given by $\calB_W \calT_N$. It is easy to show that an alternative way to derive $\mS_{N,W}$ is to consider the eigenvectors of the $N \times N$ {\em prolate matrix} $\mB_{N,W}$~\cite{Varh1993ProlateMatrix}, which is the matrix with entries given by
\begin{equation} \label{eq:Prolate matrix}
\mB_{N,W}[m,n]: = \frac{\sin2\pi W(m-n)}{\pi(m-n)}
\end{equation}
for all $m,n\in\{0,1,\ldots,N-1\}$.  Indeed, $\mB_{N,W}$ can be understood as the finite truncation of the infinite matrix representation of  $\calB_W \calT_N$. Thus, $\mS_{N,W}$ contains the eigenvectors of $\mB_{N,W}$ and we can write $\mB_{N,W}$ as
\[
\mB_{N,W} = \mS_{N,W} \mLambda_{N,W} \mS_{N,W}^*
 \]
where $\mLambda_{N,W}$ is an $N \times N$ diagonal matrix with the eigenvalues $\lambda_{N,W}^{(0)}, \ldots, \lambda_{N,W}^{(N-1)},$ along the main diagonal (sorted in descending order).

Our primary goal is to develop fast algorithms for working with $\mS_{N,W}$ (or $\mB_{N,W}$, which also arises in many practical applications, as detailed in Section~\ref{ssec:apps} below).  Towards this end, we will begin by examining the relationship between $\mB_{N,W}$ and the matrix obtained by projecting onto the lowest $2NW$ Fourier coefficients.  To be more precise, for any $f \in [-\frac12, \frac12]$ we will let
\[
\ve_f:= \left[\begin{array}{c}e^{j2\pi f0}\\e^{j2\pi f1}\\\vdots\\e^{j2\pi f(N-1)} \end{array}\right]
\]
denote a length-$N$ vector of samples from a discrete-time complex exponential signal with digital frequency $f$. We then define $W'$ such that $2NW'$ is the nearest odd integer to $W$, and we let $\mF_{N,W}$ denote the partial Fourier matrix with the lowest $2NW'$ frequency DFT vectors of length $N$, i.e.,
\begin{equation}\label{eq:PartialFourierMatrix}
\mF_{N,W} =  \dfrac{1}{\sqrt{N}}\begin{bmatrix}  \ve_{ -(2NW'-1)/2N} & \cdots & \ve_{(2NW'-1)/2N} \end{bmatrix}.
\end{equation}
Note that the projection onto the span of $\mF_{N,W}$ is given by the matrix $\mF_{N,W} \mF_{N,W}^*$, which has entries given by
\begin{equation} \label{eq:FFstar}
[\mF_{N,W} \mF_{N,W}^*][m,n] = \dfrac{1}{N}\sum_{k = -NW' +\frac12}^{NW'-\frac12} e^{j 2\pi (m-n)k/N} = \frac{\sin(\pi(2NW')\tfrac{m-n}{N})}{N\sin(\pi\tfrac{m-n}{N})} = \frac{\sin(2\pi W' (m-n))}{N\sin(\pi\tfrac{m-n}{N})}
\end{equation}
for $m,n=0, \ldots, N-1$. Comparing~\eqref{eq:Prolate matrix} with~\eqref{eq:FFstar} we see that $\mB_{N,W}$ and $\mF_{N,W} \mF_{N,W}^*$ share a somewhat similar structure, where $\mB_{N,W}$ is a Toeplitz matrix with rows (or columns) given by the sinc function, whereas $\mF_{N,W} \mF_{N,W}^*$ is a circulant matrix with rows (or columns) given by the digital sinc or Dirichlet function. In Theorem~\ref{thm:prolateFFTLR}, which is proven in Section~\ref{sec:proofProlateFFTLR}, we show that up to a small approximation error $\eps$, the difference between these two matrices has a rank of $O(\log N\log\tfrac{1}{\eps})$.

\begin{thm} \label{thm:prolateFFTLR}
Let $N \in \N$ and $W \in (0, \tfrac{1}{2})$ be given. Then for any $\eps \in (0,\tfrac{1}{2})$, there exist $N \times r_1$ matrices $\mL_1$,$\mL_2$ and an $N \times N$ matrix $\mE_1$ such that
\[
\mB_{N,W} = \mF_{N,W} \mF_{N,W}^* + \mL_1\mL_2^* + \mE_1,
\]
where
\[
r_1 \le \left( \frac{4}{\pi^2} \log(8N) + 6 \right) \log \left( \frac{15}{\eps} \right) ~~\text{and}~~ \|\mE_1\| \le \eps.
\]
\end{thm}

We also note that the proof of Theorem~\ref{thm:prolateFFTLR} provides an explicit construction such matrices $\mL_1$ and $\mL_2$, which could be of use in practice.

An important consequence of Theorem~\ref{thm:prolateFFTLR} which will be useful to us, and which is also of independent interest, is that it can be used to establish a nonasymptotic bound on the number of eigenvalues $\lambda_{N,W}^{(\ell)}$ of $\mB_{N,W}$ in the ``transition region'' between $\eps$ and $1-\eps$. In particular, Lemma~\ref{lem:SlepEig} tells as that in the limit as $N \rightarrow \infty$ we will have that the first $\approx 2NW$ eigenvalues will approach $1$ while the last $\approx N(1-2W)$ eigenvalues will approach $0$. However, this does not address precisely how many eigenvalues we can expect to find between $\eps$ and $1-\eps$.

In \cite{Slepi_ProlateV}, it is shown that for any $b \in \R$, if $k = \lfloor 2WN + \tfrac{b}{\pi}\log N\rfloor$, then $\lambda_{N,W}^{(k)} \to (1+e^{\pi b})^{-1}$ as $N \to \infty$. By setting $b = \tfrac{1}{\pi}\log(\tfrac{1}{\eps}-1)$, we get $\lambda_{N,W}^{(k)} \to \eps$. Similarly, by setting $b = -\tfrac{1}{\pi}\log(\tfrac{1}{\eps}-1)$, we get $\lambda_{N,W}^{(k)} \to 1-\eps$. Thus, for fixed $W$ and $\eps$, we get the following asymptotic result:
\begin{equation} \label{eq:asymptotic}
\#\{\ell : \eps \le \lambda_{N,W}^{(\ell)} \le 1-\eps\} \sim \frac{2}{\pi^2}\log N \log\left(\frac{1}{\eps}-1\right).
\end{equation}

In Figure~\ref{fig:EigenvalueGap} on the left, we show a numerical comparison of  $\#\{\ell : \eps \le \lambda_{N,W}^{(\ell)} \le 1-\eps\}$ and $\frac{2}{\pi^2}\log N \log\left(\frac{1}{\eps}-1\right)$ versus $N$ for a fixed value of $W = \tfrac{1}{4}$. The size of the eigenvalue gap appears to grow linearly with $\log N$ and linearly with $\log(\frac{1}{\eps}-1)$ as expected.

 In Figure~\ref{fig:EigenvalueGap} on the right, we show a plot of $\#\{\ell : \eps \le \lambda_{N,W}^{(\ell)} \le 1-\eps\}$ versus $W$ for a fixed value of $N = 2^{16}$. Note that we did not include the range $\tfrac{1}{4} < W < \tfrac{1}{2}$ due to the fact that the DPSS eigenvalues satisfy $\lambda_{N,1/2-W}^{(\ell)} = 1-\lambda_{N,W}^{(N-1-\ell)}$ (equation (13) in \cite{Slepi_ProlateV}), and thus, $\#\{\ell : \eps < \lambda_{N,1/2-W}^{(\ell)} < 1-\eps\} = \#\{\ell : \eps < \lambda_{N,W}^{(\ell)} < 1-\eps\}$. Based on this plot, the size of the eigenvalue gap appears to grow roughly linearly with respect to $\log W$ over the range $0 < W \le \tfrac{1}{4}$. None of the theoretical results capture how the size of the eigenvalue gap depends on $W$. However, in most applications $W$ is a fixed constant that is not too small, and so, the dependence with respect to $W$ is of little consequence. 

\begin{figure}
   \centering
  \includegraphics[scale = 0.38]{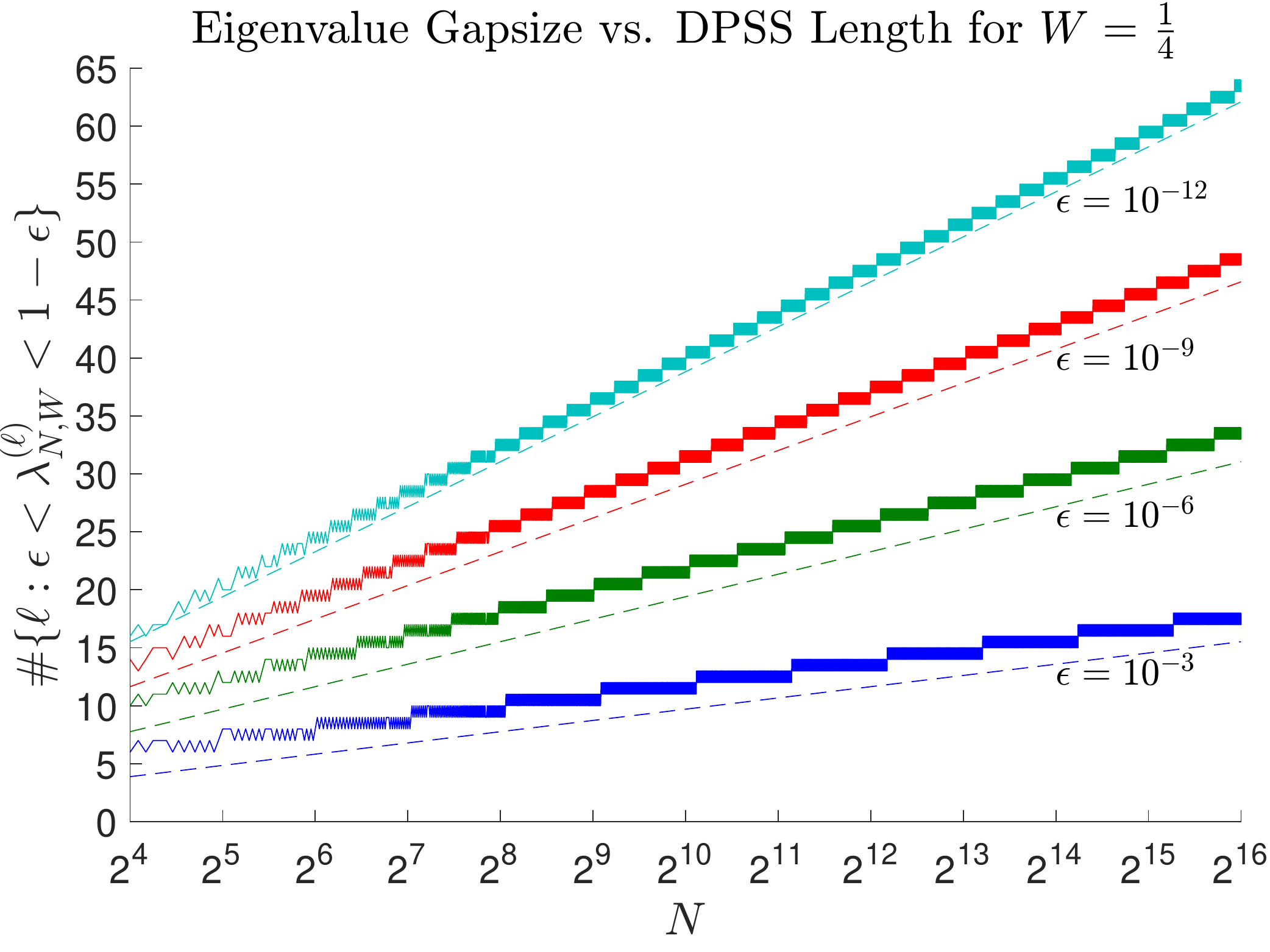} \includegraphics[scale = 0.38]{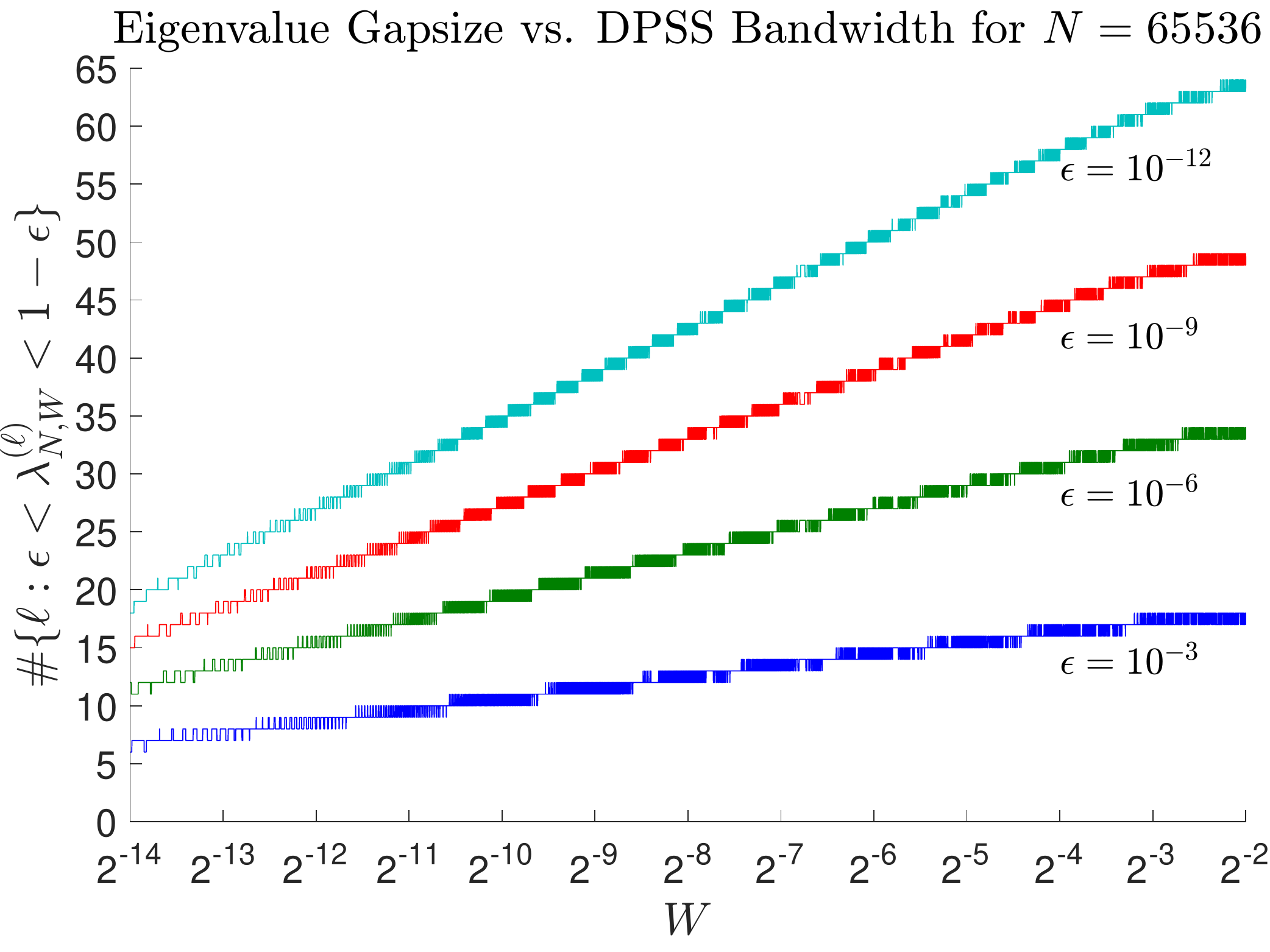}
   \caption{\small \sl (Left) The solid lines represent the size of eigenvalue gap for $2^4 \le N \le 2^{16}$, $W = \frac{1}{4}$, and $\eps = 10^{-3}, 10^{-6}, 10^{-9}, 10^{-12}$. The dashed lines represent the asymptotic result in~\eqref{eq:asymptotic}. (Right) The solid lines represent the size of eigenvalue gap for $N = 2^{16}$, $2^{-14} \le W \le \tfrac{1}{4}$, and $\eps = 10^{-3}, 10^{-6}, 10^{-9}, 10^{-12}$. 
   \label{fig:EigenvalueGap}}
\end{figure}

A nonasymptotic bound on the width of this transition region is given in~\cite{ZhuWakin2015MDPSS}, which shows that for any $N \in \N$, $W \in (0, \tfrac{1}{2})$, and $\eps \in(0,\tfrac{1}{2})$,
$$
\#\{\ell : \eps \le \lambda_{N,W}^{(\ell)} \le 1-\eps\} \le \frac{\tfrac{2}{\pi^2}\log(N-1)+\tfrac{2}{\pi^2}\tfrac{2N-1}{N-1}}{\eps(1-\eps)}.
$$
This bound correctly highlights the logarithmic dependence on $N$, but can be quite poor when $\eps$ is very small ($O(1/\eps)$ as opposed to the $O(\log(1/\eps))$ dependence in the asymptotic result). In the following corollary of Theorem~\ref{thm:prolateFFTLR}, we significantly sharpen this bound in terms of its dependence on $\eps$ to within a constant factor of the optimal asymptotic result. The intuition behind this result is that Theorem~\ref{thm:prolateFFTLR} demonstrates that $\mB_{N,W}$ can be approximated as $\mF_{N,W} \mF_{N,W}^*$ (a matrix whose eigenvalues are all either equal to 1 or 0) plus a low-rank correction, and the rank of this correction limits the number of possible eigenvalues in the transition region.

\begin{cor}\label{cor:transition}
Let $N \in \N$ and $W \in (0, \tfrac{1}{2})$ be given. Then for any $\eps \in (0,\tfrac{1}{2})$,
\[
\#\{\ell : \eps < \lambda_{N,W}^{(\ell)} < 1-\eps\} \le \left(\frac{8}{\pi^2}\log(8N)+12\right)\log\left(\frac{15}{\eps}\right).
\]
\end{cor}
This result is analogous to the main result of~\cite{Israel2015EigenvalueDisTFLocalization}, which recently established similar nonasymptotic results concerning the eigenvalue distribution of the {\em continuous} time-frequency localization operator. Since we are dealing with discrete version of the time-frequency localization operator, we are able to use different techniques to obtain a much tighter bound on the number of possible eigenvalues in the transition region.
%This result is analogous to the main result of~\cite{Israel2015EigenvalueDisTFLocalization}, which recently established similar nonasymptotic results concerning the eigenvalue distribution of the {\em continuous} time-frequency localization operator. In fact, while our approach is quite a bit different than that of~\cite{Israel2015EigenvalueDisTFLocalization}, it is also possible to establish a version of Corollary~\ref{cor:transition} (with different constants) using some of the same proof techniques as~\cite{Israel2015EigenvalueDisTFLocalization}.

Finally, we describe a few additional consequences of these results. Recall that $\mB_{N,W} = \mS_{N,W} \mLambda_{N,W} \mS_{N,W}^*$. From Corollary~\ref{cor:transition} we have that the diagonal entries of the matrix $\mLambda_{N,W}$ are mostly very close to 1 or 0, with only a small number of eigenvalues lying in between.  Thus, recalling that $\mS_K$ denotes the $N \times K$ matrix containing the first $K$ elements of the Slepian basis $\mS_{N,W}$, it is reasonable to expect that $\mB_{N,W}$ and $\mS_K \mS_K^*$ (the matrix obtained by setting the top $K$ eigenvalues to 1 and the remainder to 0) should be within a low-rank correction when $K\approx 2NW$.  The following corollary shows that this is indeed the case.

\begin{cor}\label{cor:prolateSlepianLR}
Let $N \in \N$ and $W \in (0, \tfrac{1}{2})$ be given.  For any $\eps \in (0,\tfrac{1}{2})$, fix $K$ to be such that $\lambda_{N,W}^{(K-1)} > \eps$ and $\lambda_{N,W}^{(K)} < 1-\eps$. Then there exist $N \times r_2$ matrices $\mU_1, \mU_2$ and an $N \times N$ matrix $\mE_2$ such that
\[
\mS_K \mS_K^* = \mB_{N,W} + \mU_1\mU_2^* + \mE_2,
\]
where
\[
r_2 \le \left( \frac{8}{\pi^2} \log(8N) + 12 \right) \log \left( \frac{15}{\eps} \right) ~~ \text{and} ~~  \| \mE_2 \| \le \eps.
\]
\end{cor}

Similarly, consider the rank-$K$ truncated pseudoinverse of $\mB_{N,W}$ where $K \approx 2NW$ (which we will denote by $\mB_{N,W}^{\dagger}$). Since most of the first $K$ eigenvalues of $\mB_{N,W}$ are very close to $1$, most of the first $K$ eigenvalues of $\mB_{N,W}^{\dagger}$ will also be close to $1$. Also, most of the last $N-K$ eigenvalues of $\mB_{N,W}$ are very close to $0$, and by definition the last $N-K$ eigenvalues of $\mB_{N,W}^{\dagger}$ are exactly $0$. Hence, it is reasonable to expect that $\mB_{N,W}$ and $\mB_{N,W}^{\dagger}$ are within a low-rank correction when $K \approx 2NW$. The following corollary shows that this is indeed the case.

\begin{cor}\label{cor:prolatePseudoinverseLR}
Let $N \in \N$ and $W \in (0, \tfrac{1}{2})$ be given.  For any $\eps \in (0,\tfrac{1}{2})$, fix $K$ to be such that $\lambda_{N,W}^{(K-1)} > \eps$ and $\lambda_{N,W}^{(K)} < 1-\eps$, and let $\mB_{N,W}^{\dagger}$ be the rank-$K$ truncated pseudoinverse of $\mB_{N,W}$. Then there exist $N \times r_3$ matrices $\mU_3, \mU_4$ and an $N \times N$ matrix $\mE_3$ such that
\[
\mB_{N,W}^{\dagger} = \mB_{N,W} + \mU_3 \mU_4^* + \mE_3,
\]
where
\[
r_3 \le \left( \frac{8}{\pi^2} \log(8N) + 12 \right) \log \left( \frac{15}{\eps} \right) ~~ \text{and} ~~  \| \mE_3 \| \le 3\eps.
\]
\end{cor}
A similar decomposition of the pseudoinverse of $\mB_{N,W}$, also based on the sum of the prolate matrix and a low rank update, was presented in \cite{huybrechs2016FastAlgorithms}.  Our result above gives an explicit non-asymptotic bound on the rank of the update required to achieve a certain accuracy.

Also, consider the matrix $\mB_{N,W}^{(\text{tik})} = (\mB_{N,W}^2+\alpha\mId)^{-1}\mB_{N,W}$ where $\alpha > 0$. (Note that this matrix is associated with Tikhonov regularization, i.e. for a given $\vy \in \C^N$, the vector $\vx \in \C^N$ which minimizes $\|\vy-\mB_{N,W}\vx\|_2^2+\alpha\|\vx\|_2^2$ is given by $\vx = (\mB_{N,W}^2+\alpha\mId)^{-1}\mB_{N,W}\vy = \mB_{N,W}^{(\text{tik})}\vy$). Since most of the eigenvalues of $\mB_{N,W}$ are either very close to $1$ or very close to $0$, most of the eigenvalues of $\mB_{N,W}^{(\text{tik})}$ are either very close to $\tfrac{1}{1+\alpha}$ or very close to $0$. Hence, it is reasonable to expect that $\tfrac{1}{1+\alpha}\mB_{N,W}$ and $\mB_{N,W}^{(\text{tik})}$ are within a low-rank correction. The following corollary shows that this is indeed the case.

\begin{cor}\label{cor:prolateTikhonovLR}
Let $N \in \N$ and $W \in (0, \tfrac{1}{2})$ and $\alpha > 0$ be given, and define $\mB_{N,W}^{(\text{tik})} = (\mB_{N,W}^2+\alpha \mId)^{-1}\mB_{N,W}$. Then, for any $\eps \in (0,\tfrac{1}{2})$, there exists an $N \times r_4$ matrix $\mU_5$ and an $N \times N$ matrix $\mE_4$ such that
\[
\mB_{N,W}^{(\text{tik})} = \dfrac{1}{1+\alpha}\mB_{N,W} + \mU_5 \mU_5^* + \mE_4,
\]
where
\[
r_4 \le \left( \frac{8}{\pi^2} \log(8N) + 12 \right) \log \left( \frac{15}{\min(\alpha(1+\alpha)\eps,\tfrac{1}{3}\eps)} \right) ~~ \text{and} ~~  \| \mE_4 \| \le \eps.
\]
\end{cor}

In the next section, we will use Theorem ~\ref{thm:prolateFFTLR} along with Corollaries ~\ref{cor:prolateSlepianLR}, ~\ref{cor:prolatePseudoinverseLR}, and ~\ref{cor:prolateTikhonovLR} to derive fast algorithms for working with the Slepian basis.

\subsection{The Fast Slepian Transform}

\subsection*{A fast factorization of $\mS_K\mS_K^*$}

Suppose we wish to compress a vector $\vx \in \C^N$ of $N$ uniformly spaced samples of a signal down to a vector of $K \approx 2NW$ elements in such a way that best preserves the DTFT of the signal over $|f| \le W$. We can do this by storing $\mS_K^*\vx$, which is a vector of $K < N$ elements, and then later recovering $\mS_K\mS_K^*\vx$, which contains nearly all of the energy of the signal in the frequency band  $|f| \le W$. However, na\"{i}ve multiplication of $\mS_K$ or $\mS_K^*$ takes $O(NK) = O(2WN^2)$ operations. For certain applications, this may be intractable.

If we combine the results of Corollary~\ref{cor:prolateSlepianLR} along with that of Theorem~\ref{thm:prolateFFTLR}, we get that
\begin{align*}
\mS_K\mS_K &= \mB_{N,W} + \mU_1\mU_2^* + \mE_2
\\
&= \mF_{N,W}\mF_{N,W}^* + \mL_1\mL_2^* + \mU_1\mU_2^* + \mE_1 + \mE_2
\\
&= \mT_1\mT_2^* + \mE_1 + \mE_2
\end{align*}
where $$\mT_1 = \begin{bmatrix}\mF_{N,W} & \mL_1 & \mU_1 \end{bmatrix} ~~~~~ \text{and} ~~~~~ \mT_2 = \begin{bmatrix}\mF_{N,W} & \mL_2 & \mU_2 \end{bmatrix}.$$
Both $\mT_1$ and $\mT_2$ are $N \times K'$ matrices where $$K' = 2NW'+r_1+r_2 \le \lceil2NW\rceil + \left( \dfrac{12}{\pi^2} \log(8N) + 18 \right) \log \left( \dfrac{15}{\eps} \right).$$ So we can compress $\vx$ by computing $\mT_2^*\vx$, which is a vector of $K' \approx 2NW$ elements, and then later recover $\mT_1\mT_2^*\vx$. By using the triangle inequality, we have $\|\mS_K\mS_K^*-\mT_1\mT_2^*\| = \|\mE_1+\mE_2\| \le \|\mE_1\|+\|\mE_2\| \le 2\eps$. Hence, $\|\mS_K\mS_K^*\vx-\mT_1\mT_2^*\vx\|_2 \le 2\eps\|\vx\|_2$ for any vector $\vx \in \C^N$. Both $\mF_{N,W}$ and $\mF_{N,W}^*$ can be applied to a vector in $O(N\log N)$ operations via the FFT. Since $\mL_1$, $\mL_2$, $\mU_1$, and $\mU_2$ are $N \times O(\log N \log\tfrac{1}{\eps})$ matrices, $\mL_1$, $\mL_2^*$, $\mU_1$, and $\mU_2^*$ can each be applied to a vector in $O(N\log N \log\tfrac{1}{\eps})$ operations. Therefore, computing $\mT_2^*\vx$ and later recovering $\mT_1\mT_2^*\vx$ (as an approximation for $\mS_K\mS_K^*\vx$) takes $O(N\log N \log\tfrac{1}{\eps})$ operations.

\vspace{0.2cm}

\subsection*{Fast projections onto the range of $\mS_K$}

Alternatively, if we only require computing the projected vector $\mS_K\mS_K^*\vx$, and compression is not required, then there is a simpler solution. Corollary~\ref{cor:prolateSlepianLR} tells us that $\|\mS_K\mS_K^*-(\mB_{N,W}+\mU_1\mU_2^*)\| \le \eps$, and thus, $\|\mS_K\mS_K^*\vx-(\mB_{N,W}\vx+\mU_1\mU_2^*\vx)\|_2 \le \eps\|\vx\|_2$ for any vector $\vx \in \C^N$. Since $\mB_{N,W}$ is a Toeplitz matrix, computing $\mB_{N,W}\vx$ can be done in $O(N\log N)$ operations via the FFT. Since $\mU_1$ and $\mU_2$ are $N \times O(N \log N \log \tfrac{1}{\eps})$ matrices, computing $\mU_1\mU_2^*\vx$ can be done in $O(N \log N \log \tfrac{1}{\eps})$ operations. Therefore, we can compute $\mB_{N,W}\vx+\mU_1\mU_2^*\vx$ as an approximation to $\mS_K\mS_K^*\vx$ using only $O(N \log N \log \tfrac{1}{\eps})$ operations.

\vspace{0.2cm}

\subsection*{Fast rank-$K$ truncated pseudoinverse of $\mB_{N,W}$}

A closely related problem to working with the matrix $\mS_K \mS_K^*$ concerns the task of solving a linear system of the form $\vy = \mB_{N,W}\vx$. Since the prolate matrix has several eigenvalues that are close to $0$, the system is often solved by using the rank-$K$ truncated pseudoinverse of $\mB_{N,W}$ where $K \approx 2NW$. Even if the pseudoinverse is precomputed and factored ahead of time, it still takes $O(NK) = O(2WN^2)$ operations to apply to a vector $\vy \in \C^N$. Corollary~\ref{cor:prolatePseudoinverseLR} tells us that $\|\mB_{N,W}^{\dagger}-(\mB_{N,W}+\mU_3\mU_4^*)\| \le 3\eps$, and thus, $\|\mB_{N,W}^{\dagger}\vy-(\mB_{N,W}\vy+\mU_3\mU_4^*\vy)\|_2 \le 3\eps\|\vy\|_2$ for any vector $\vy \in \C^N$. Again, computing $\mB_{N,W}\vy$ can be done in $O(N \log N)$ operations using the FFT, and computing $\mU_3\mU_4^*\vy$ can be done in $O(N\log N \log \tfrac{1}{\eps})$ operations. Therefore, we can compute $\mB_{N,W}\vy+\mU_3\mU_4^*\vy$ as an approximation to $\mB_{N,W}^{\dagger}\vy$ using only $O(N \log N \log \tfrac{1}{\eps})$ operations.

\vspace{0.2cm}

\subsection*{Fast Tikhonov regularization involving $\mB_{N,W}$}

Another approach to solving the ill-conditioned system $\vy = \mB_{N,W}\vx$ is to use Tikhonov regularization, i.e., minimize $\|\vy-\mB_{N,W}\vx\|_2^2+\alpha\|\vx\|_2^2$ where $\alpha > 0$ is a regularization parameter. The solution to this minimization problem is $\vx = (\mB_{N,W}^2+\alpha\mId)^{-1}\mB_{N,W}\vy$. Even if the matrix $\mB_{N,W}^{(\text{tik})} = (\mB_{N,W}^2+\alpha\mId)^{-1}\mB_{N,W}$ is computed ahead of time, it still takes $O(N^2)$ operations to apply to a vector $\vy$. Corollary~\ref{cor:prolateTikhonovLR} tells us that $\|\mB_{N,W}^{(\text{tik})}-(\mB_{N,W}+\mU_5\mU_5^*)\| \le \eps$, and thus, $\|\mB_{N,W}^{(\text{tik})}\vy-(\mB_{N,W}\vy+\mU_5\mU_5^*\vy)\|_2 \le \eps\|\vy\|_2$ for any vector $\vy \in \C^N$. Again, computing $\mB_{N,W}\vy$ can be done in $O(N \log N)$ operations via the FFT. Since $\mU_5$ is a $N \times O\left(\log N \max(\log \frac{1}{\alpha\eps} , \log\frac{1}{\eps})\right)$ matrix, computing $\mU_5\mU_5^*\vy$ can be done in $O\left(N \log N \max(\log \frac{1}{\alpha\eps} , \log\frac{1}{\eps})\right)$ operations. Therefore, we can compute $\mB_{N,W}\vy+\mU_5\mU_5^*\vy$ as an approximation to $\mB_{N,W}^{(\text{tik})}\vy$ using only $O\left(N \log N \max(\log \frac{1}{\alpha\eps} , \log\frac{1}{\eps})\right)$ operations.

\vspace{0.2cm}

The least-squares problems above involve the inverse of $\mB_{N,W}$, a symmetric semi-definite Toeplitz matrix.  There is a long history of ``superfast'' algorithms for inverting such systems in the signal processing \cite{morf74fa,kailath79di} and numerical linear algebra \cite{ammar87ge,heinig95in,vanbarel03su} literature. These algorithms take a number of different forms.  They usually work by breaking the matrix into smaller blocks, either hierarchically \cite{martinsson05fa} or recusively \cite{bitmead80as,pan01st}, and then exploiting the structure of the matrix to efficiently combine the solutions of smaller systems into a solution for the entire system. The overall computational complexity of these algorithms is $O(N\log^2 N)$ for the first solve with a given matrix, and $O(N\log N)$ for subsequent solves.  An overview of these methods can be found in \cite{turnes14ef}.

The approach suggested by Corollary~\ref{cor:prolatePseudoinverseLR} (and the regularized version in Corollary~\ref{cor:prolateTikhonovLR}) have the same run time of $O(N\log N)$, but are based on entirely different principles.  Theorem~\ref{thm:prolateFFTLR} essentially states that the matrix $\mB_{N,W}$ is a low-rank update away from an orthoprojection, and this orthoprojection can be computed quickly using the FFT. Corollaries ~\ref{cor:prolatePseudoinverseLR} and ~\ref{cor:prolateTikhonovLR} show that this property also holds for the (regularized) pseudo-inverse.  These mathematical results show that this particular system can be very closely approximated  by a sum of circulant and low-rank matrices, which leads directly to efficient algorithms for solving least-squares problems.

It is worth noting that the algorithms above also have a fast precomputation time. The columns of $\boldsymbol{U}_1$, $\boldsymbol{U}_2$, $\boldsymbol{U}_3$, $\boldsymbol{U}_4$, and $\boldsymbol{U}_5$ are all rescaled Slepian basis vectors. In \cite{huybrechs2016FastAlgorithms}, the authors describe a method to compute $r$ Slepian basis vectors and corresponding eigenvalues in $O(rN\log N)$ operations. Hence, $\boldsymbol{U}_1$, $\boldsymbol{U}_2$, $\boldsymbol{U}_3$, and $\boldsymbol{U}_4$ can be precomputed in $O(N \log^2 N \log \tfrac{1}{\epsilon})$ operations and $\boldsymbol{U}_5$ can be precomputed in $O(N \log^2 N \max(\log \frac{1}{\alpha \epsilon } , \log \frac{1}{\epsilon}))$ operations. Also, the proof of Theorem~\ref{thm:prolateFFTLR} provides a construction of $\boldsymbol{L}_1$ and $\boldsymbol{L}_2$ which can be computed in $O(N \log N \log \tfrac{1}{\epsilon})$ operations. Hence, the fast factorization of $\boldsymbol{S}_K\boldsymbol{S}_K^*$, the fast projection onto the range of $\boldsymbol{S}_K$, and the fast truncated pseudoinverse of $\boldsymbol{B}_{N,W}$ have a precomputation time of $O(N\log^2 N \log \tfrac{1}{\epsilon})$, and the fast Tikhonov regularization has a precomputation time of $O(N \log^2 N \max (\log \frac{1}{\alpha \epsilon } , \log \frac{1}{\epsilon}))$. 

\subsection{Applications}
\label{ssec:apps}

Owing to the concentration in the time and frequency domains, the Slepian basis vectors have proved to be useful in numerous signal processing problems~\cite{AhmadQianAmin2015WallCluterDPSS,DavenportWakin2012CSDPSS,Slepi_ProlateV,zemen2005channelEstim,DavenSSBWB_Wideband}. Linear systems of equations involving the prolate matrix $\mB_{N,W}$ also arise in several problems, such as band-limited extrapolation~\cite{Slepi_ProlateV}. In this section, we describe some specific applications that stand to benefit from the fast constructions described above.

%\note{We may wish to add some additional applications.}

\vspace{0.2cm}
{\bf $i.$ Representation and compression of sampled bandlimited and multiband signals.} Consider a length-$N$ vector $\vx$ obtained by uniformly sampling a baseband analog signal $x(t)$
%$$x(t) = \int_{-B_{\text{band}}/2}^{B_{\text{band}}/2} X(F)e^{j2\pi Ft} dF$$
over the time interval $[0,NT_s)$ with sampling period $T_s\leq \frac{1}{B_{\text{band}}}$ chosen to satisfy the Nyquist sampling rate. Here, %$X(F)$ denotes the {\em continuous-time Fourier transform} (CTFT) of $x(t)$ and is assumed to be supported on a narrow band $[-B_{\text{band}}/2,B_{\text{band}}/2]$.
$x(t)$ is assumed to be bandlimited with frequency range $[-B_{\text{band}}/2,B_{\text{band}}/2]$. Under this assumption, the sample vector $\vx$ can be expressed as
\begin{equation}
\vx[n] = \int_{-W}^{W} X(f) e^{j2\pi fn} \; df, ~n = 0,1,\dots, N-1,
\label{eq:sampled bandlimited vectors}
\end{equation}
or equivalently,
\begin{equation}
\vx = \int_{-W}^{W} X(f) \ve_f \; df
\label{eq:sampled bandlimited vectors ef}
\end{equation}
where $W = T_sB_{\text{band}}/2\leq \frac12$ and $X(f)$ is the DTFT of the infinite sample sequence $x[n] = x(nT_s), n\in \mathbb{Z}$. Such finite-length vectors of samples from bandlimited signals arise in problems such as time-variant channel estimation \cite{zemen2005channelEstim} and mitigation of narrowband interference \cite{Davenport2010SignalProcessingCompressiveMeasurenemts}. Solutions to these and many other problems benefit from representations that efficiently capture the structure inherent in vectors $\vx$ of the form \eqref{eq:sampled bandlimited vectors ef}.

In~\cite{DavenportWakin2012CSDPSS}, the authors showed that such a vector $\vx$ has a low-dimensional structure by building a dictionary in which $\vx$ can be approximated with a small number of atoms. The $N\times N$ DFT basis is insufficient to capture the low dimensional structure in $\vx$  due to the ``DFT leakage'' phenomenon. In particular, the DFT basis is comprised of vectors $\ve_f$ with $f$ sampled uniformly between $-1/2$ and $1/2$. From \eqref{eq:sampled bandlimited vectors ef}, one can interpret $\vx$ as being comprised of a linear combination of vectors $\ve_f$ with $f$ ranging continuously between $-W$ and $W$. It is natural to ask whether $\vx$ could be efficiently represented using only the DFT vectors $\ve_f$ with $f$ between $-W$ and $W$; in particular, these are the columns of the matrix $\mF_{N,W}$ defined in \eqref{eq:PartialFourierMatrix}. Unfortunately, this is not the case---while a majority of the energy of $\vx$ can be captured using the columns of $\mF_{N,W}$, a nontrivial amount will be missed and this is contained in the familiar sidelobes in the DFT outside the band of interest.

An efficient alternative to the partial DFT $\mF_{N,W}$ is given by the partial Slepian basis $\mS_K$ when $K \approx 2NW$. In~\cite{DavenportWakin2012CSDPSS}, for example, it is established that when $\vx$ is generated by sampling a bandlimited analog random process with flat power spectrum over $[-B_{\text{band}}/2,B_{\text{band}}/2]$, and when one chooses $K = 2NW(1+\rho)$, then on average $\mS_K \mS_K^\ast \vx$ will capture all but an exponentially small amount of the energy from $\vx$. Zemen and Mecklenbr\"{a}uke \cite{zemen2005channelEstim} showed that expressing the time-varying subcarrier coefficients in a Slepian basis yields better performance than that obtained with a DFT basis, which suffers from frequency leakage.

By modulating the (baseband) Slepian basis vectors to different frequency bands and then merging these dictionaries, one can also obtain a new dictionary that offers an efficient representation of sampled {\em multiband} signals. Zemen et al.~\cite{zemen2007minimum} proposed one such dictionary for estimating a time-variant flat-fading channel whose spectral support is a union of several intervals. In the context of compressive sensing, Davenport and Wakin~\cite{DavenportWakin2012CSDPSS} investigated multiband modulated DPSS dictionaries for sparse recovery of sampled multiband signals, and Sejdi\'{c} et al.~\cite{sejdic2012compressive} applied such dictionaries for recovery of physiological signals from compressive measurements.
Zhu and Wakin~\cite{Zhu2015targetDetectDPSS} employed such dictionaries for detecting targets behind the wall in through-the-wall radar imaging, and modulated DPSS's can also be useful for mitigating wall clutter~\cite{AhmadQianAmin2015WallCluterDPSS}.

In summary, many of the above mentioned problems are facilitated by projecting a length-$N$ vector onto the subspace spanned by the first $K \approx 2NW$ Slepian basis vectors (i.e., computing $\mS_K \mS_K^\ast \vx$). One version of the Block-Based CoSaMP algorithm in \cite{DavenportWakin2012CSDPSS} involves computing the projection of a vector onto the column space of a modulated DPSS dictionary. The channel estimates proposed in~\cite{SejdicICASSP2008ChannelEstimationDPSS} are based on the projection of the subcarrier coefficients onto the column space of the modulated multiband DPSS dictionary. Of course, one can also compress $\vx$ by keeping the $K \approx 2NW$ Slepian basis coefficients $\mS_K^\ast \vx$ instead of the $N$ entries of $\vx$. Computationally, all of these problems benefit from having a fast Slepian transform: whereas direct matrix-vector multiplication would require $O(NK) = O(2WN^2)$ operations, the fast Slepian constructions allow these computations to be approximated in only $O\left(N \log N \log \frac{1}{\eps}\right)$ operations.

%\vspace{0.2cm}
%{\bf $ii.$ Projection onto the subspace spanned by the first $\approx 2NW$ DPSS vectors:} many of the above mentioned problems are indeed facilitated by projecting a length-$N$ vector onto the subspace spanned by the  first $\approx 2NW$ DPSS vectors. One version of Algorithm 2 (Block-Based CoSaMP) in \cite{DavenportWakin2012CSDPSS} involves computing the projection of a vector onto the column space of a modulated DPSS dictionary.  The channel estimates proposed in~\cite{SejdicICASSP2008ChannelEstimationDPSS}  are based on the projection of the subcarrier coefficients onto the column space of modulated multiband DPSS dictionary \noteZ{Can we merge this item with item $i$ or just delete this item?}

\vspace{0.2cm}
{\bf $ii.$ Prolate matrix linear systems.} Linear equations of the form $\mB_{N,W}\vy = \vb$ arise naturally in signal processing. For example, suppose we obtain the length-$N$ sampled bandlimited vector $\vx$ as defined in (\ref{eq:sampled bandlimited vectors}) and we are interested in estimating the infinite-length sequence $x[n] = x(nT_s),~\forall n\in\mathbb{Z}$. The discrete-time signal $x[n]$ is assumed to be bandlimited to $[-W,W]$ for $W<\frac{1}{2}$. Let $\mathcal{I}_N: \ell_2(\mathbb{Z})\rightarrow \mathbb{C}^N$ denote the index-limiting operator that restricts a sequence to its entries on $[N]$ (and produces a vector of length $N$); that is $\mathcal{I}_N(y)[m]: = y[m]$ for all $m\in\{0,1,\ldots,N-1\}$.
%The adjoint operator $\mathcal{I}^*_N: \mathbb{C}^N\rightarrow \ell_2(\mathbb{Z})$ (anti-index-limiting operator) is given by
%$$
%\mathcal{I}^*_N(\vy)[m]: = \left\{ \begin{array}{ll} \vy[m], & m\in\{0,1,\ldots,N-1\}\\
%0, & \text{otherwise}. \end{array}\right.
%$$
Also, recall that $\mathcal{B}_{W}: \ell_2(\mathbb{Z})\rightarrow \ell_2(\mathbb{Z})$ denotes the band-limiting operator that bandlimits the DTFT of a discrete-time signal to the frequency range $[-W,W]$.
%, i.e., for $y\in \ell_2(\mathbb{Z})$, we have that
%$$
% \mathcal{B}_{W}(y)[m]: = \int_{-W}^{W}e^{j2\pi fm}df\star y[m]=\sum_{n=-\infty}^{\infty}\left(\frac{\sin(2\pi W(m-n))}{\pi (m-n)}y[n] \right)
%$$
%where $\star$ stands for convolution.
Given $\vx$, the least-squares estimate $\widehat x[n]\in l_2(\mathbb{Z})$ for the infinite-length bandlimited sequence takes the form
$$\widehat{x}[n] = [(\mathcal{I}_N\mathcal{B}_W)^\dagger\vx][n] = \sum_{m=0}^{N-1}v[m]\frac{\sin2\pi W(n-m)}{\pi(n-m)},$$
where $\vv = \mB_{N,W}^\dagger \vx$.

Another problem involves linear prediction of bandlimited signals based on past samples. Suppose $x(t)$ is a continuous, zero-mean, wide sense stationary random process with power spectrum
\begin{equation} P_{x}(F)=\left\{\begin{array}{ll} \frac{1}{B_{\textup{band}}},& - \frac{B_{\textup{band}}}{2}\leq F \leq  \frac{B_{\textup{band}}}{2}, \\0, & \text{otherwise}.\end{array}\right.\nonumber\end{equation}
Let $x[n]=x(nT_s)$ denote the samples acquired from $x(t)$ with a sampling interval of $T_s\leq \frac{1}{B_{\text{band}}}$. A linear prediction of $x[N]$ based on the previous $N$ samples $x[0],x[1],\ldots,x[N-1]$ takes the form~\cite{Slepi_ProlateV}
$$\widehat x[N] = \sum_{n=0}^{N-1} a_n x[n].$$
Choosing $a_n$ such that $\widehat x[N]$ has the minimum mean-squared error is equivalent to solving
$$\min_{a_n} \varrho :=\E\left[\left(\sum_{n=0}^{N-1} a_n x[n] - x[N]\right)^2  \right].$$
Let $W = \frac{T_s}{2B_{\text{band}}}$. Taking the derivative of $\varrho$ and setting it to zero  yields
\begin{align*}
\mB_{N,W}\va = \vb
\end{align*}
with $\va = \left[ a_0 ~ a_1 ~ \cdots ~ a_{N-1} \right]^T$ and $\vb = \left[\frac{\sin \left(2\pi W N  \right)}{\pi N } ~  \frac{\sin \left(2\pi W \left(N-1 \right) \right)}{\pi \left(N -1  \right)} ~ \cdots~ \frac{\sin \left(2\pi W 1 \right)}{\pi 1}\right]^T$.  Thus the optimal $\widehat{\va}$ is simply given by $\widehat{\va} = \mB_{N,W}^\dagger \vb$.

We present one more example: the Fourier extension~\cite{huybrechs2010fourierExtension}. The partial Fourier series sum
$$y_{N'} (t) = \frac{1}{\sqrt{2}}\sum_{|n|\leq N'}\widehat{y}_ne^{jn\pi t},~~\widehat{y}_n = \frac{1}{\sqrt 2}\int_{-1}^{1}y(t)e^{-jn\pi t}dt$$
of a non-periodic function $y\in L^2_{[-1,1]}$ (such as $y(t) = t$) suffers from the Gibbs phenomenon. One approach to overcome the Gibbs phenomenon is to extend the function $y$ to a function $g$ that is periodic on a larger interval $[-T,T]$ with $T>1$ and compute the partial Fourier series of $g$ \cite{huybrechs2010fourierExtension}. Let $G_{N''}$ be the space of bandlimited $2T$-periodic functions
$$ G_{N''} := \left\{g: g(t) = \sum_{n=-N''}^{N''}{\widehat g}_ne^{\frac{jn\pi t}{T}},~\widehat g_n\in \mathbb{C}\right\}.$$
The Fourier extension problem involves finding
\begin{equation}
g_{N''} := \arg\min_{g\in G_{N''}} \| y - g \|_{L^2_{[-1,1]}}.
\label{eq:FourierExtensionProb}\end{equation}
The solution $g_{N''}$ is called the Fourier extension of $y$ to the interval $[-T,T]$. Let $\widehat{\vg} = \left[ \widehat g_{-N''} ~ \cdots ~ \widehat g_{0} ~ \cdots ~ \widehat g_{N''}  \right]^T$ and define $\mathcal{F}_{N''}: L_2([-1,1])\rightarrow \mathbb{C}^{2N''+1}$ as
$$(\mathcal{F}_{N''}(u))[n] = \frac{1}{\sqrt{2T}}\int_{-1}^{1}u(t)e^{-\frac{jn\pi t}{T}}dt , ~~|n|\leq N''.$$
For convenience, here we index all vectors and matrices beginning at $-N''$. Any minimizer $\widehat{\vg}$ of the least-squares problem~\eqref{eq:FourierExtensionProb} must satisfy the normal equations
\begin{equation}\mathcal{F}_{N''}\mathcal{F}_{N''}^*\widehat{\vg} = \mathcal{F}_{N''}y,\label{eq:fourier extenion linear problem}\end{equation}
where $\mathcal{F}_{N''}y$ can be efficiently approximately computed via the FFT. One can show that $\mathcal{F}_{N''}\mathcal{F}_{N''}^* = \mB_{N,W}$, where $N = 2N'' +1$ and $W = \frac{1}{2T}\leq \frac{1}{2}$.

Each of the above least-squares problems could be solved by computing a rank-$K$ truncated pseudo-inverse of $\mB_{N,W}$ with $K \approx 2NW$. Direct multiplication of this matrix with a vector, however, would require $O(NK) = O(2WN^2)$ operations. The fast methods we have developed allow a fast approximation to the truncated pseudo-inverse to be applied in only $O\left(N \log N \log \frac{1}{\eps}\right)$ operations.

This paper continues in Sections~\ref{sec:proofProlateFFTLR} and \ref{sec:proofCorollaries} with a proof of Theorem~\ref{thm:prolateFFTLR} and its corollaries. We conclude the paper in Section~\ref{sec:simulations} with several simulations that demonstrate the computational advantages of the proposed constructions.

%\section{DPSS Background}
%\input{dpssBackground}

\section{Proof of Theorem~\ref{thm:prolateFFTLR}}
\label{sec:proofProlateFFTLR}
Our goal is to show that $\mB_{N,W} - \mF_{N,W} \mF_{N,W}^*$ is well-approximated as a factored low rank matrix. To do this, we will express $\mB_{N,W} - \mF_{N,W} \mF_{N,W}^*$ in terms of other matrices, whose entries also have a closed form. We will then derive a factored low rank approximation for each of these other matrices. Finally, we will combine these low rank approximations to get a factored low rank approximation for $\mB_{N,W} - \mF_{N,W} \mF_{N,W}^*$.

Towards this end, we define $\mD_A$ to be the $N \times N$ diagonal matrix with diagonal entries $\mD_A[n,n] = e^{j2\pi W'n}$ for $n=0, \ldots, N-1$ and define $\mA_0$ to be the $N \times N$ matrix with entries
\[
\mA_0[m,n] = \begin{cases} \frac{1}{\pi (m-n)} - \frac{1}{N\sin\left(\pi \tfrac{m-n}{N}\right)} & \text{if} \ m \neq n, \\ 0 & \text{if} \ m = n.\end{cases}
\]
We also define $\mD_B$ to be the $N \times N$ diagonal matrix with diagonal entries $\mD_B[n,n] = e^{j(W+W')n}$ for $n = 0,\ldots,N-1$ and define $\mB_0$ to be the $N \times N$ matrix with entries
\[
\mB_0[m,n] = \dfrac{2\sin(\pi(W-W')(m-n))}{\pi(m-n)}.
\]
Note that with these definitions, we can write the $(m,n)$-th entry of $\mB_{N,W} - \mF_{N,W} \mF_{N,W}^*$ as
\begin{align*}
&[\mB_{N,W} - \mF_{N,W} \mF_{N,W}^*] [m,n]
\\
&= \frac{\sin(2\pi W(m-n))}{\pi(m-n)} - \frac{\sin(2\pi W' (m-n))}{N\sin(\pi\tfrac{m-n}{N})}
\\
&= \frac{\sin(2\pi W(m-n))}{\pi(m-n)} - \frac{\sin(2\pi W'(m-n))}{\pi(m-n)} + \frac{\sin(2\pi W'(m-n))}{\pi(m-n)}- \frac{\sin(2\pi W' (m-n))}{N\sin(\pi\tfrac{m-n}{N})}
\\
&= \frac{2\sin(\pi(W-W')(m-n))\cos(\pi(W+W')(m-n))}{\pi(m-n)} + \frac{\sin(2\pi W'(m-n))}{\pi(m-n)}- \frac{\sin(2\pi W' (m-n))}{N\sin(\pi\tfrac{m-n}{N})}
\\
&= \mB_0[m,n]\cos(\pi(W+W')(m-n)) + \mA_0[m,n]\sin(2\pi W'(m-n))
\\
%&= \dfrac{1}{2}e^{j\pi(W+W')m}\mB_0[m,n]e^{-j\pi(W+W')n} + \dfrac{1}{2}e^{-j\pi(W+W')m}\mB_0[m,n]e^{j\pi(W+W')n} 
%\\
%&~~~~~+ \dfrac{1}{2j}e^{j2\pi W'm}\mA_0[m,n]e^{-j2\pi W'n} - \dfrac{1}{2j}e^{-j2\pi W'm}\mA_0[m,n]e^{j2\pi W'n}
%\\
&= \left[\frac{1}{2}\mD_B\mB_0\mD_B^* + \frac{1}{2}\mD_B^*\mB_0\mD_B + \frac{1}{2j}\mD_A\mA_0\mD_A^* - \frac{1}{2j}\mD_A^*\mA_0\mD_A\right][m,n].
\end{align*}
Hence,
\begin{equation}\label{eq:BA0}
\mB_{N,W} - \mF_{N,W} \mF_{N,W}^* = \frac{1}{2j}\mD_A\mA_0\mD_A^* - \frac{1}{2j}\mD_A^*\mA_0\mD_A + \frac{1}{2}\mD_B\mB_0\mD_B^* + \frac{1}{2}\mD_B^*\mB_0\mD_B.
\end{equation}
Thus, we can find a low-rank approximation for $\mB_{N,W} - \mF_{N,W} \mF_{N,W}^*$  by finding low-rank approximations for $\mA_0$ and $\mB_0$. In order to do so, it is useful to consider the $N \times N$ matrix $\mA_1$ defined by
\[
\mA_1[m,n] = \begin{cases} \mA_0[m,n] - \frac{1}{\pi(m-n+N)} - \frac{1}{\pi(m-n-N)} & \text{if} \ m \neq n, \\ 0 & \text{if} \ m = n.\end{cases}
\]
Next, we let $\mH$ denote the {\em Hilbert matrix}, i.e., the $N \times N$ matrix with entries
\[
 \mH[m,n] = \frac{1}{m+n+1},
\]
and let $\mJ$ be the $N \times N$ matrix with $1$'s along the antidiagonal and zeros elsewhere. (Note that for an arbitrary $N \times N$ matrix $\mX$, $\mJ \mX$ is simply $\mX$ flipped vertically and $\mX \mJ$ is $\mX$ flipped horizontally.) Using these definitions, we can write $\mA_0$ as
\begin{equation} \label{eq:A0A1}
\mA_0 = \frac{1}{\pi}(\mH \mJ - \mJ \mH) + \mA_1 .
\end{equation}
By combining~\eqref{eq:BA0} and~\eqref{eq:A0A1}, we get
\begin{equation} \label{eq:BFF}
\begin{aligned}
&\mB_{N,W} - \mF_{N,W} \mF_{N,W}^*
\\
&= \tfrac{1}{2j}\mD_A\left[\tfrac{1}{\pi}(\mH \mJ - \mJ \mH) + \mA_1 \right]\mD_A^* - \tfrac{1}{2j}\mD_A^*\left[\tfrac{1}{\pi}(\mH \mJ - \mJ \mH) + \mA_1 \right]\mD_A + \tfrac{1}{2}\mD_B\mB_0\mD_B^* + \tfrac{1}{2}\mD_B^*\mB_0\mD_B.~~~~~~~~
\end{aligned}
\end{equation}
Therefore, we can come up with a factored low rank approximation for $\mB_{N,W} - \mF_{N,W} \mF_{N,W}^*$ by first deriving factored low rank approximations for each of the matrices $\mH$, $\mA_1$, and $\mB_0$.

\subsubsection*{Low rank approximation of $\mH$}

Our goal is to construct a low-rank matrix $\widetilde{\mH} = \mZ \mZ^*$ such that $\| \mH - \widetilde{\mH} \| \le \delta_H$ for some desired $\delta_H > 0$.  We will do this via Lemma~\ref{lem:Lyapunov}, which we prove in the appendix.

\begin{lemma} \label{lem:Lyapunov}
Let $\mA$ be an $N\times N$ symmetric positive definite matrix with condition number $\kappa$, let $\mB$ be an arbitrary $N\times M$ matrix with $M\leq N$, and let $\mX$ be the $N\times N$ positive definite solution to the Lyapunov equation
\[
\mA\mX + \mX\mA^* = \mB\mB^*.
\]	
Then for any $\delta \in (0,1]$, there exists an $N\times rM$ matrix $\mZ$ with
\begin{equation}\label{eq:rankH}
r =\left\lceil\frac{1}{\pi^2}\log\left(4\kappa\right)\log\left(\frac{4}{\delta}\right) \right\rceil,
\end{equation}
such that
\begin{equation}
\label{eq:lylrap}
\|\mX-\mZ\mZ^*\| ~\leq~ \delta \|\mX\|.
\end{equation}
\end{lemma}

Now, we let $\mA$ be the $N \times N$ diagonal matrix defined by $\mA[n,n] = n+\tfrac{1}{2}$ for $n = 0,\ldots,N-1$, and let $\mB \in \R^{N}$ be a vector of all ones. It is easy to verify that the positive definite solution $\mX$ to $\mA \mX + \mX \mA^* = \mB \mB^*$ is simply $\mX = \mH$.
The minimum and maximum eigenvalues of $\mA$ are $\lambda_{min}(\mA) = \tfrac{1}{2}$ and $\lambda_{max}(\mA) = N-\tfrac{1}{2}$, and thus the condition number for $\mA$ is $\kappa = 2N-1$.  Thus, by applying Lemma~\ref{lem:Lyapunov} with $\delta = \frac{\delta_H}{\pi}$, we can construct an $N \times r_H$ matrix $\mZ$ with
\[
r_H = \ceil{ \frac{1}{\pi^2} \log(8N - 4) \log\left(\frac{4 \pi}{\delta_H}\right)}
\]
such that
\[
\| \mH  - \mZ \mZ^* \| \le \frac{\delta_H}{\pi} \| \mH \|.
\]
It is shown in \cite{schur1911} that the operator norm of the infinite Hilbert matrix is bounded above by $\pi$, and thus, the finite dimensional matrix $\mH$ satisfies $\|\mH\| \le \pi$. Therefore, $\|\mH-\mZ\mZ^*\| \le \delta_H$, as desired.

\subsubsection*{Low rank approximation of $\mA_1$}
Next, we construct a low-rank matrix $\widetilde{\mA}_1$ such that $\| \mA_1 - \widetilde{\mA}_1 \| \le \delta_A$ for some desired $\delta_A > 0$. In this case we will require a different approach. We begin by noting that by using the Taylor series expansions \footnote{Here, $\zeta(s) := \sum_{n = 1}^{\infty}n^{-s}$ is the Riemann-Zeta function.}
\[
\frac{1}{\sin \pi x} - \frac{1}{\pi x} = \frac{2}{\pi}\sum_{k = 1}^{\infty}(1-2^{-(2k-1)})\zeta(2k)x^{2k-1},
\]
and
\[
\frac{1}{\pi(x+1)} + \frac{1}{\pi(x-1)} = -\frac{2x}{\pi(1-x^2)} = -\frac{2}{\pi}\sum_{k = 1}^{\infty}x^{2k-1},
\]
we can write
\begin{align*}
\mA_1[m,n] &= \frac{1}{\pi (m-n)} - \frac{1}{N\sin\left(\pi \tfrac{m-n}{N}\right)} - \frac{1}{\pi(m-n+N)} - \frac{1}{\pi(m-n-N)} \\
&= \frac{2}{N\pi}\sum_{k = 1}^{\infty}\left[1-(1-2^{-(2k-1)})\zeta(2k)\right]\left(\frac{m-n}{N}\right)^{2k-1}.
\end{align*}
We can then define a new $N \times N$ matrix $\widetilde{\mA}_1$ by truncating the series to $r_A$ terms:
\[
\widetilde{\mA}_1[m,n] := \frac{2}{N\pi}\sum_{k = 1}^{r_A}\left[1-(1-2^{-(2k-1)})\zeta(2k)\right]\left(\frac{m-n}{N}\right)^{2k-1}.
\]
Note that each entry of $\widetilde{\mA}_1$ is a polynomial of degree $2r_A-1$ in both $m$ and $n$.  Thus, we could also write
\[
\widetilde{\mA}_1[m,n] = \sum_{k=0}^{2r_A-1} \sum_{\ell=0}^{2r_A-1} c_{k,\ell} m^k n^{\ell},
\]
for a set of scalars $c_{k,\ell} \in \R$. If we let $\mV_A$ be the $N \times 2r_A$ matrix with entries $\mV_A[m,k] = m^k$ and let $\mC_A$ be the $2r_A \times 2r_A$ matrix with entries $\mC_A[k,\ell] = c_{k,\ell}$, it is easy to see that we can write $\widetilde{\mA}_1 = \mV_A \mC_A \mV_A^*.$ Thus, $\widetilde{\mA}_1$ has rank $2r_A$.

Next, we note that by using the identity $(1-2^{1-s})\zeta(s) = \sum_{n = 1}^{\infty}\frac{(-1)^{n+1}}{n^s}$ for $s > 1$, we have
\[
0 \le 1-(1-2^{-(2k-1)})\zeta(2k) = \sum_{n = 2}^{\infty}\frac{(-1)^{n}}{n^{2k}} \le \frac{1}{2^{2k}},
\]
where the inequality follows from the fact that this is an alternating series whose terms decrease in magnitude. Hence, we can bound the truncation error $|\mA_1[m,n] - \widetilde{\mA}_1[m,n]|$ by
\begin{align*}
|\mA_1[m,n] - \widetilde{\mA}_1[m,n]| &= \left|\frac{2}{N\pi}\sum_{k = r_A+1}^{\infty}\left[1-(1-2^{-(2k-1)})\zeta(2k)\right]\left(\frac{m-n}{N}\right)^{2k-1}\right| \\
&\le \frac{2}{N\pi}\sum_{k = r_A+1}^{\infty}\left[1-(1-2^{-(2k-1)})\zeta(2k)\right]\left|\frac{m-n}{N}\right|^{2k-1} \\
&\le \frac{2}{N\pi}\sum_{k = r_A+1}^{\infty}\frac{1}{2^{2k}} \cdot 1  \\
&= \frac{2}{3N\pi}\left(\frac{1}{2}\right)^{2r_A}.
\end{align*}
Therefore, the error $\|\mA_1 - \widetilde{\mA}_1\|_F^2$ is bounded by:
\[
\|\mA_1 - \widetilde{\mA}_1\|_F^2 = \sum_{m = 0}^{N-1}\sum_{n = 0}^{N-1}|\mA_1[m,n] - \widetilde{\mA}_1[m,n]|^2  \le N \cdot N \cdot \left[\frac{2}{3N\pi}\left(\frac{1}{2}\right)^{2r_A}\right]^2 = \frac{4}{9\pi^2}\left(\frac{1}{2}\right)^{4r_A}.
\]
Hence,
\[
\|\mA_1 - \widetilde{\mA}_1\| \le \|\mA_1 - \widetilde{\mA}_1\|_F \le \frac{2}{3\pi}\left(\frac{1}{2}\right)^{2r_A}.
\]
Thus, for any $\delta_A \in (0, \tfrac{8}{3 \pi})$ we can set $r_A = \ceil{\frac{1}{2\log 2}\log\left(\frac{2}{3\pi \delta_A}\right)}$ and ensure that $\| \mA_1 - \widetilde{\mA}_1\| \le \delta_A$ and
\[
\rank(\widetilde{\mA}_1) \le 2 \ceil{\frac{1}{2\log 2}\log\left(\frac{2}{3\pi \delta_A}\right)}.
\]

\subsubsection*{Low rank approximation of $\mB_0$}
To construct a low-rank matrix $\widetilde{\mB}_0$ such that $\| \mB_0 - \widetilde{\mB}_0 \| \le \delta_B$ for some desired $\delta_B > 0$, we use a similar approach as above. Using the Taylor series
$$
\sin x = \sum_{k = 0}^{\infty}\frac{(-1)^{k}x^{2k+1}}{(2k+1)!},
$$
we can write
\begin{align*}\mB_0[m,n] &= \frac{2\sin(\pi(W-W')(m-n))}{\pi(m-n)} \\ &= \sum_{k = 0}^{\infty}\frac{2(-1)^k[\pi(W-W')(m-n)]^{2k+1}}{(2k+1)!\pi(m-n)} \\ &= \frac{2}{N\pi}\sum_{k = 0}^{\infty}\frac{(-1)^k[\pi(W-W')N]^{2k+1}}{(2k+1)!}\left(\frac{m-n}{N}\right)^{2k}.
\end{align*}
We can then define a new matrix $\widetilde{\mB}_0$ by truncating the series to $r_B$ terms:
\[
\widetilde{\mB}_0[m,n] := \frac{2}{N\pi}\sum_{k = 0}^{r_B-1}\frac{(-1)^k[\pi(W-W')N]^{2k+1}}{(2k+1)!}\left(\frac{m-n}{N}\right)^{2k}.
 \]
Note that each entry of $\widetilde{\mB}_0$ is a polynomial of degree $2r_B-2$ in both $m$ and $n$. Thus, we could also write
\[
\widetilde{\mB}_0[m,n] = \sum_{k=0}^{2r_B-2} \sum_{\ell=0}^{2r_B-2} c'_{k,\ell} m^k n^{\ell},
\]
for a set of scalars $c'_{k,\ell} \in \R$. If we let $\mV_B$ be the $N \times (2r_B-1)$ matrix with entries $\mV_B[m,k] = m^k$ and let $\mC_B$ be the $(2r_B-1) \times (2r_B-1)$ matrix with entries $\mC_B[k,\ell] = c'_{k,\ell}$, it is easy to see that we can write $\widetilde{\mB}_0 = \mV_B \mC_B \mV_B^*.$ Thus, $\widetilde{\mB}_0$ has rank $2r_B-1$.

Next, we note that by definition, $2W'N$ is $2WN$ rounded to the nearest odd integer. Hence, $|2W'N-2WN| \le 1$, and so, $|\pi(W-W')N| \le \frac{\pi}{2}$. Also,  we have that $(2k+1)! \ge \frac{2}{9} \cdot 3^{2k+1}$ for all integers $k \ge 0$. Hence, the truncation error $|\mB_0[m,n] - \widetilde{\mB}_0[m,n]|$ is bounded by:
\begin{align*}
|\mB_0[m,n] - \widetilde{\mB}_0[m,n]| & \le \left|\frac{2}{N\pi}\sum_{k = r_B}^{\infty}\frac{(-1)^k[\pi(W-W')N]^{2k+1}}{(2k+1)!}\left(\frac{m-n}{N}\right)^{2k}\right| \\
&\le \frac{2}{N\pi}\left|\frac{(-1)^{r_B}[\pi(W-W')N]^{2r_b+1}}{(2K+1)!}\left(\frac{m-n}{N}\right)^{2K}\right| \\
&\le \frac{2}{N\pi}\frac{\left(\tfrac{\pi}{2}\right)^{2r_B+1}}{\tfrac{2}{9} \cdot 3^{2r_B+1}} \cdot 1 \\
&= \frac{3}{2N}\left(\frac{\pi}{6}\right)^{2r_B},
\end{align*}
where we have used the fact that an alternating series whose terms decrease in magnitude can be bounded by the magnitude of the first term. Thus, the error $\| \mB_0 - \widetilde{\mB}_0\|_F^2$ is bounded by:
\[
\|\mB_0 - \widetilde{\mB}_0\|_F^2 = \sum_{m = 0}^{N-1}\sum_{n = 0}^{N-1}|\mB_0[m,n] - \widetilde{\mB}_0[m,n]|^2 \le N \cdot N \cdot \left[\frac{3}{2N}\left(\frac{\pi}{6}\right)^{2r_B}\right]^2 = \frac{9}{4}\left(\frac{\pi}{6}\right)^{4r_B}.
\]
Hence,
\[
\|\mB_0 - \widetilde{\mB}_0\| \le \|\mB_0 - \widetilde{\mB}_0\|_F \le \frac{3}{2}\left(\frac{\pi}{6}\right)^{2r_B}.
\]
Thus, for any $\delta_B \in (0, \tfrac{3}{2})$ we can set $r_B = \ceil{\frac{1}{2\log \tfrac{6}{\pi}}\log\left(\frac{3}{2 \delta_B}\right)}$ and ensure that $\| \mB_0 - \widetilde{\mB}_0\| \le \delta_B$ and
\[
\rank(\widetilde{\mB}_0) \le 2 \ceil{\frac{1}{2\log \tfrac{6}{\pi}}\log\left(\frac{3}{2\delta_B}\right)}-1.
\]

\subsubsection*{Putting it all together}
Now that we have a way to construct a factored low rank approximation of $\mH$, $\mA_1$, and $\mB_0$, we will combine those results to derive a factored low rank approximation for $\mB_{N,W}-\mF_{N,W}\mF_{N,W}^*$. For any $\eps \in (0,\tfrac{1}{2})$, set\footnote{It may be possible to obtain a slightly better bound via a more careful selection of $\delta_A$, $\delta_B$, and $\delta_H$. We have not pursued such refinements here as there is not much room for significant improvement.} $\delta_H = \frac{4\pi}{15}\eps$ and $\delta_A = \delta_B = \frac{7}{30}\eps$. Then, let $\widetilde{\mH} = \mZ \mZ^*$, $\widetilde{\mA}_1 = \mV_A\mC_A\mV_A^*$, and $\widetilde{\mB}_0 = \mV_B\mC_B\mV_B^*$ be defined as in the previous subsections. Also, define $\mE_{H} = \mH - \widetilde{\mH}$, $\mE_{A} = \mA_1 - \widetilde{\mA}_1$, $\mE_{B} = \mB_0 - \widetilde{\mB}_0$. By using these definitions along with~\eqref{eq:BFF}, we can write
$$\mB_{N,W} - \mF_{N,W}\mF_{N,W}^* = \mL + \mE_1,$$
where
\begin{align*}
\mL &= \tfrac{1}{2j}\mD_A\left[\tfrac{1}{\pi}(\widetilde{\mH} \mJ - \mJ \widetilde{\mH}) + \widetilde{\mA}_1 \right]\mD_A^* - \tfrac{1}{2j}\mD_A^*\left[\tfrac{1}{\pi}(\widetilde{\mH} \mJ - \mJ\widetilde{\mH}) + \widetilde{\mA_1} \right]\mD_A + \tfrac{1}{2}\mD_B\widetilde{\mB}_0\mD_B^* + \tfrac{1}{2}\mD_B^*\widetilde{\mB}_0\mD_B
\\
&= \tfrac{1}{2\pi j}(\mD_A\mZ\mZ^*\mJ\mD_A^* - \mD_A\mJ\mZ\mZ^*\mD_A^* - \mD_A^*\mZ\mZ^*\mJ\mD_A - \mD_A^*\mJ\mZ\mZ^*\mD_A)
\\
&~~~~~+\tfrac{1}{2j}(\mD_A\mV_A\mC_A\mV_A^*\mD_A^* - \mD_A^*\mV_A\mC_A\mV_A^*\mD_A) + \tfrac{1}{2}(\mD_B\mV_B\mC_B\mV_B^*\mD_B^* + \mD_B^*\mV_B\mC_B\mV_B^*\mD_B)
\end{align*}
and
$$\mE_1 = \tfrac{1}{2j}\mD_A\left[\tfrac{1}{\pi}(\mE_H \mJ - \mJ \mE_H) + \mE_A \right]\mD_A^* - \tfrac{1}{2j}\mD_A^*\left[\tfrac{1}{\pi}(\mE_H \mJ - \mJ \mE_H) + \mE_A \right]\mD_A + \tfrac{1}{2}\mD_B\mE_B\mD_B^* + \tfrac{1}{2}\mD_B^*\mE_B\mD_B.$$
If we define
$$\mL_1 = \begin{bmatrix}\tfrac{1}{2\pi j}\mD_A\mZ & -\tfrac{1}{2\pi j}\mD_A\mJ\mZ & -\tfrac{1}{2\pi j}\mD_A^*\mZ & \tfrac{1}{2\pi j}\mD_A^*\mJ\mZ & \tfrac{1}{2j}\mD_A\mV_A & -\tfrac{1}{2j}\mD_A^*\mV_A & \tfrac{1}{2}\mD_B\mV_B & \tfrac{1}{2}\mD_B^*\mV_B\end{bmatrix}$$
and
$$\mL_2 = \begin{bmatrix}\mD_A\mJ\mZ & \mD_A\mZ & \mD_A^*\mJ\mZ & \mD_A^*\mZ & \mD_A\mV_A\mC_A^* & \mD_A^*\mV_A\mC_A^* & \mD_B\mV_B\mC_B^* & \mD_B^*\mV_B\mC_B^*\end{bmatrix},$$
then $\mL = \mL_1\mL_2^*$ and $\mL_1, \mL_2$ are both $N \times r_1$ matrices, where
\begin{align*}
r_1 &= 4 \cdot r_H+2 \cdot 2r_A+2 \cdot (2r_B-1)
\\
& = 4\ceil{\dfrac{1}{\pi^2}\log(8N-4)\log\left(\dfrac{4\pi}{\delta_H}\right)}+4\ceil{\dfrac{1}{2\log 2}\log\left(\dfrac{2}{3\pi\delta_A}\right)}+4\ceil{\dfrac{1}{2\log\tfrac{6}{\pi}}\log\left(\dfrac{3}{2\delta_B}\right)}-2
\\
&\le \dfrac{4}{\pi^2}\log(8N-4)\log\left(\dfrac{4\pi}{\delta_H}\right) + \dfrac{2}{\log 2}\log\left(\dfrac{2}{3\pi\delta_A}\right) + \dfrac{2}{\log\tfrac{6}{\pi}}\log\left(\dfrac{3}{2\delta_B}\right) + 10
\\
&= \dfrac{4}{\pi^2}\log(8N-4)\log\left(\dfrac{15}{\eps}\right) + \dfrac{2}{\log 2}\log\left(\dfrac{20}{7\pi \eps}\right) + \dfrac{2}{\log \tfrac{6}{\pi}}\log\left(\dfrac{45}{7 \eps}\right) + 10
\\
&= \left(\dfrac{4}{\pi^2}\log(8N-4) + \frac{2}{\log 2} + \frac{2}{\log \tfrac{6}{\pi}}\right)\log\left(\frac{15}{\eps}\right) + \dfrac{2}{\log 2}\log\left(\dfrac{4}{21\pi}\right) + \dfrac{2}{\log \tfrac{6}{\pi}}\log\left(\dfrac{3}{7}\right) + 10
\\
&\le \left(\dfrac{4}{\pi^2}\log(8N)+6\right)\log\left(\dfrac{15}{\eps}\right).
\end{align*}
Also, by applying the triangle inequality and using the fact that  $\|\mD_A\| = \|\mD_B\| = \|\mJ\| = 1$, we see that
\[
\|\mE_1\| \le \frac{2}{\pi}\|\mE_H\| + \|\mE_A\| + \|\mE_B\| \le \frac{2}{\pi}\delta_H + \delta_A + \delta_B= \frac{2}{\pi} \cdot \frac{4\pi\eps}{15} + \frac{7\eps}{30} + \frac{7\eps}{30} = \eps.
\]
Together, these two facts establish the theorem. \qed

\section{Proofs of Corollaries}
\label{sec:proofCorollaries}
\subsection{Proof of Corollary~\ref{cor:transition}}
Corollary~\ref{cor:transition} is a direct consequence of Theorem~\ref{thm:prolateFFTLR} together with the following lemma.

\begin{lemma}
Let $\mA$ be an $N\times N$ Hermitian matrix with eigenvalues $\lambda^{(0)}\geq\cdots\geq\lambda^{(N-1)}$.  Suppose we can write
	\[
		\mA = \mU\mU^* + \mL + \mE,
	\]
	where $\mU$ is an $N \times K$ matrix with orthonormal columns ($\mU^*\mU=\mId$), $\mL$ is an $N \times N$ Hermitian matrix with $\rank(\mL)=r$, and $\mE$ is an $N \times N$ Hermitian matrix with $\|\mE\|\leq\epsilon$.  Then
	\[
		\#\{\ell~:~\epsilon<\lambda^{(\ell)} < 1-\epsilon\} ~\leq~ 2r.
	\]
\end{lemma}
\begin{proof}
Define $\setS_1 = \Null(\mU\mU^*)\cap\Null(\mL)$ and $d_1=\dimension(\setS_1)$.  For any $\vx\in\setS_1$ with $\|\vx\|_2=1$,
\begin{align*}
	\vx^*\mA\vx &= \vx^*\mE\vx \leq \|\mE\|\leq\epsilon.
\end{align*}
Then by the Courant-Fischer-Weyl min-max theorem,
\[
	\lambda^{(N-d_1)} = \min_{\substack{\text{subspaces $\setS$} \\ \dimension(\setS)=d_1}}
	\left[\max_{\substack{\vx\in\setS \\ \|\vx\|_2=1}}\vx^*\mA\vx\right]
	~\leq~
	\max_{\substack{\vx\in\setS_1 \\ \|\vx\|_2=1}}\vx^*\mA\vx \leq \epsilon.
\]
Similarly, let $\setS_2 = \Null(\mId-\mU\mU^*)\cap\Null(\mL)$ and $d_2 = \dimension(\setS_2)$. Then, for any $\vx\in\setS_2$ with $\|\vx\|_2 = 1$,
\[
	\vx^*(\mId-\mA)\vx = -\vx^*\mE\vx \leq \|\mE\| \le \epsilon.
\]
Since the eigenvalues of $\mId-\mA$ are $1-\lambda^{(N-1)}\geq \cdots\geq 1-\lambda^{(0)}$, the min-max theorem tells us
\[
	1-\lambda^{(d_2-1)} = \min_{\substack{\text{subspaces $\setS$} \\ \dimension(\setS)=d_2}}
	\left[\max_{\substack{\vx\in\setS \\ \|\vx\|_2=1}}\vx^*(\mId-\mA)\vx\right]
	~\leq~
	\max_{\substack{\vx\in\setS_2 \\ \|\vx\|_2=1}}\vx^*(\mId-\mA)\vx \leq \epsilon,
\]
meaning $\lambda^{(d_2-1)}\geq1-\epsilon$.  Thus
\begin{align*}
	\#\{\ell~:~\epsilon < \lambda^{(\ell)} < 1-\epsilon\} &\leq N - d_1 -  d_2.
\end{align*}
Since for any two $N\times N$ matrices $\mM_1,\mM_2$
\[
	\dimension(\Null(\mM_1)\cap\Null(\mM_2)) \geq N - (\dimension(\Range(\mM_1)) + \dimension(\Range(\mM_2))),
\]
we know $d_1\geq N-(K+r)$ and $d_2\geq N-(N-K+r)$.  Thus,
\[
	\#\{\ell~:~\epsilon < \lambda^{(\ell)} < 1-\epsilon\} ~\leq~ 2r.
\]

\end{proof}

\subsection{Proof of Corollary~\ref{cor:prolateSlepianLR}}
\label{ssec:ProlSlep}
Corollary~\ref{cor:prolateSlepianLR} is a direct consequence of Corollary~\ref{cor:transition} together with the following lemma.
\begin{lemma} \label{lem:prolateSlepianLR}
Let $\mA$ be an $N\times N$ symmetric matrix with eigenvalues $1\geq\lambda^{(0)}\geq\cdots\geq\lambda^{(N-1)}\geq 0$ and corresponding eigenvectors $\vv_0,\ldots,\vv_{N-1}$.  Fix $\eps \in (0,\tfrac{1}{2})$, and let
\[
r' = \#\{\ell~:~\epsilon < \lambda^{(\ell)} < 1-\epsilon\}.
\]
Choose $K$ such that $\lambda^{(K-1)} > \epsilon$ and $\lambda^{(K)} < 1-\epsilon$, and set $\mV_{[K]} = \begin{bmatrix} \vv_0 & \cdots & \vv_{K-1}\end{bmatrix}$.  Then there exist $N\times r'$ matrices $\mU_1,\mU_2$ and an $N\times N$ matrix $\mE$ with $\|\mE\| \le \eps$ such that
	\[
		\mV_{[K]}\mV_{[K]}^*  = \mA + \mU_1\mU_2^* + \mE.
	\]
\end{lemma}

\begin{proof}
First, we partition the eigenvalues of $\mA$ into four sets:
\begin{align*}
	\setI_1 = \{\ell: \lambda^{(\ell)} \ge 1-\epsilon\}, \quad
	\setI_2 = \{\ell: \ell<K, ~ \epsilon < \lambda^{(\ell)} < 1-\epsilon\}, \\
	\setI_3 = \{\ell:\ell\geq K,~\epsilon < \lambda^{(\ell)} < 1-\epsilon\}, \quad
	\setI_4 = \{\ell:\lambda^{(\ell)}\leq \epsilon\}. \qquad
\end{align*}
We can write
\[
\mA = \mV_1\mLambda_1\mV_1^* + \mV_2\mLambda_2\mV_2^* + \mV_3\mLambda_3\mV_3^* + \mV_4\mLambda_4\mV_4^*
\]
and
\[
\mV_{[K]}\mV_{[K]}^* = \mV_1\mV_1^*+\mV_2\mV_2^*,
\]
where the $\mV_i$ contain the eigenvectors from $\setI_i$ as their columns, and the $\mLambda_i$ are diagonal containing the corresponding eigenvalues. Thus,
\begin{align*}
	\mV_{[K]}\mV_{[K]}^*-\mA &= \left[\mV_2(\mId-\mLambda_2)\mV_2^* - \mV_3\mLambda_3\mV_3^*\right] +
	\left[\mV_1(\mId-\mLambda_1)\mV_1^* - \mV_4\mLambda_4\mV_4^*\right] \\
	&=: \mU_1\mU_2^* + \mE,
\end{align*}
where
$$\mU_1 = \begin{bmatrix} \mV_2(\mId-\mLambda_2)^{1/2} & -\mV_3\mLambda_3^{1/2}\end{bmatrix},$$
$$\mU_2 = \begin{bmatrix} \mV_2(\mId-\mLambda_2)^{1/2} & \mV_3\mLambda_3^{1/2}\end{bmatrix},$$
$$\mE = \begin{bmatrix}\mV_1 & \mV_4\end{bmatrix}\begin{bmatrix}\mId-\mLambda_1 & \mzero \\ \mzero & -\mLambda_4\end{bmatrix}\begin{bmatrix}\mV_1^* \\ \mV_4^*\end{bmatrix}.$$
Notice that the number of columns in both $\mU_1$ and $\mU_2$ is the same as the size of $\setI_2\cup\setI_3$, which is exactly $r'$.
Also, since both $\|\mId-\mLambda_1\|\leq\epsilon$ and $\|\mLambda_4\|\leq\epsilon$, and $\begin{bmatrix}\mV_1 & \mV_4\end{bmatrix}$ has orthonormal columns, we have $\|\mE\|\leq\epsilon$.
\end{proof}

\subsection{Proof of Corollary~\ref{cor:prolatePseudoinverseLR}}
Corollary~\ref{cor:prolatePseudoinverseLR} is a direct consequence of Corollary~\ref{cor:transition} together with the following lemma.

\begin{lemma} \label{lem:prolatepinvLR}
Let $\mA$ be an $N\times N$ symmetric matrix with eigenvalues $1\geq\lambda^{(0)}\geq\cdots\geq\lambda^{(N-1)}\geq 0$ and corresponding eigenvectors $\vv_0,\ldots,\vv_{N-1}$.  Fix $\eps \in (0,\tfrac{1}{2})$, and let
\[
r = \#\{\ell~:~\epsilon < \lambda^{(\ell)} < 1-\epsilon\}.
\]
Choose $K$ such that $\lambda^{(K-1)} > \epsilon$ and $\lambda^{(K)} < 1-\epsilon$. Let $\mV_{[K]} = \begin{bmatrix} \vv_0 & \cdots & \vv_{K-1}\end{bmatrix}$ and $\mLambda_{[K]} = \text{diag}(\lambda^{(0)},\ldots,\lambda^{(K-1)})$. Define $\mA_K^{\dagger} = \mV_{[K]}\mLambda_{[K]}^{-1}\mV_{[K]}$ to be the rank-$K$ truncated pseudoinverse of $\mA$. Then there exist $N\times r$ matrices $\mU_3,\mU_4$ and an $N\times N$ matrix $\mE$ with $\|\mE\| \le 3\eps$ such that
	\[
		\mA_K^{\dagger}  = \mA + \mU_3\mU_4^* + \mE.
	\]
\end{lemma}

\begin{proof}
We partition the eigenvalues of $\mA$ into four sets:
\begin{align*}
	\setI_1 = \{\ell: \lambda^{(\ell)} \ge 1-\epsilon\}, \quad
	\setI_2 = \{\ell: \ell<K, ~ \epsilon < \lambda^{(\ell)} < 1-\epsilon\}, \\
	\setI_3 = \{\ell:\ell\geq K,~\epsilon < \lambda^{(\ell)} < 1-\epsilon\}, \quad
	\setI_4 = \{\ell:\lambda^{(\ell)}\leq \epsilon\}. \qquad
\end{align*}
We can write
\[
\mA = \mV_1\mLambda_1\mV_1^* + \mV_2\mLambda_2\mV_2^* + \mV_3\mLambda_3\mV_3^* + \mV_4\mLambda_4\mV_4^*,
\]
and
\[
\mA_K^{\dagger} = \mV_1\mLambda_1^{-1}\mV_1^*+\mV_2\mLambda_2^{-1}\mV_2^*
\]
where the $\mV_i$ contain the eigenvectors from $\setI_i$ as their columns, and the $\mLambda_i$ are diagonal containing the corresponding eigenvalues. Thus,
\begin{align*}
	\mA_K^{\dagger}-\mA &= \left[\mV_2(\mLambda_2^{-1}-\mLambda_2)\mV_2^* - \mV_3\mLambda_3\mV_3^*\right] +
	\left[\mV_1(\mLambda_1^{-1}-\mLambda_1)\mV_1^* - \mV_4\mLambda_4\mV_4^*\right] \\
	&=: \mU_3\mU_4^* + \mE,
\end{align*}
where
$$\mU_3 = \begin{bmatrix} \mV_2(\mLambda_2^{-1}-\mLambda_2)^{1/2} & -\mV_3\mLambda_3^{1/2}\end{bmatrix},$$
$$\mU_4 = \begin{bmatrix} \mV_2(\mLambda_2^{-1}-\mLambda_2)^{1/2} & \mV_3\mLambda_3^{1/2}\end{bmatrix},$$
$$\mE = \begin{bmatrix}\mV_1 & \mV_4\end{bmatrix}\begin{bmatrix}\mLambda_1^{-1}-\mLambda_1 & \mzero \\ \mzero & -\mLambda_4\end{bmatrix}\begin{bmatrix}\mV_1^* \\ \mV_4^*\end{bmatrix}.$$
Notice that the number of columns in both $\mU_3$ and $\mU_4$ is the same as the size of $\setI_2\cup\setI_3$, which is exactly $r$.
Also, since $\|\mLambda_1^{-1}-\mLambda_1\|\leq \tfrac{1}{1-\eps} - (1-\eps) = \tfrac{\eps}{1-\eps}+\eps \le 3\eps$ and $\|\mLambda_4\|\leq\epsilon$, and $\begin{bmatrix}\mV_1 & \mV_4\end{bmatrix}$ has orthonormal columns, we have $\|\mE\|\leq3\epsilon$.
\end{proof}

\subsection{Proof of Corollary~\ref{cor:prolateTikhonovLR}}
Corollary~\ref{cor:prolateTikhonovLR} is a direct consequence of Corollary~\ref{cor:transition} together with the following lemma.
\begin{lemma} \label{lem:prolatetikhonovLR}
Let $\mA$ be an $N\times N$ symmetric matrix with eigenvalues $1\geq\lambda_0\geq\cdots\geq\lambda_{N-1}\geq 0$ and corresponding eigenvectors $\vv_0,\ldots,\vv_{N-1}$. For a given regularization parameter $\alpha > 0$, define $\mA_{\text{tik}} = (\mA^*\mA+\alpha\mId)^{-1}\mA^*$. Fix $\eps \in (0,\tfrac{1}{2})$ and let $$r = \#\{\ell : \alpha(1+\alpha)\eps < \lambda_{\ell} < 1-\tfrac{1}{3}\eps\}.$$ Then there exists an $N \times r$ matrix $\mU_5$ and an $N \times N$ matrix $\mE$ with $\|\mE\| \le\eps$ such that $$\mA_{\text{tik}} = \dfrac{1}{1+\alpha}\mA + \mU_5\mU_5^* + \mE.$$
\end{lemma}

\begin{proof}
We partition the eigenvalues of $\mA$ into two sets:
$$\setI_1 = \{\ell : \alpha(1+\alpha)\eps < \lambda_{\ell} < 1-\tfrac{1}{3}\eps\}, \quad \setI_2 = \{\ell : \lambda_{\ell} \le \alpha(1+\alpha)\eps \ \text{or} \ \lambda_{\ell} \ge 1-\tfrac{1}{3}\eps\}.$$
We can write
\[
\mA = \mV_1\mLambda_1\mV_1^* + \mV_2\mLambda_2\mV_2^*
\]
and
\[
\mA_{\text{tik}} = \mV_1(\mLambda_1^2+\alpha\mId)^{-1}\mLambda_1\mV_1^*+\mV_2(\mLambda_2^2+\alpha\mId)^{-1}\mLambda_2\mV_2^*,
\]
where the $\mV_i$ contain the eigenvectors from $\setI_i$ as their columns, and the $\mLambda_i$ are diagonal containing the corresponding eigenvalues. Thus,
\begin{align*}
	\mA_{\text{tik}}-\tfrac{1}{1+\alpha}\mA &= \mV_1\left[(\mLambda_1^2+\alpha\mId)^{-1}\mLambda_1-\tfrac{1}{1+\alpha}\mLambda_1\right]\mV_1^* + \mV_2\left[(\mLambda_2^2+\alpha\mId)^{-1}\mLambda_2-\tfrac{1}{1+\alpha}\mLambda_2\right]\mV_2^* \\
	&=: \mU_5\mU_5^* + \mE,
\end{align*}
where
$$\mU_5 = \mV_1\left[(\mLambda_1^2+\alpha\mId)^{-1}\mLambda_1-\tfrac{1}{1+\alpha}\mLambda_1\right]^{1/2},$$
$$\mE =\mV_2\left[(\mLambda_2^2+\alpha\mId)^{-1}\mLambda_2-\tfrac{1}{1+\alpha}\mLambda_2\right]\mV_2^*.$$
Notice that the number of columns in $\mU_5$ is the same as the size of $\setI_1$, which is exactly $r$.

Observe that the matrix $(\mLambda_2^2+\alpha\mId)^{-1}\mLambda_2-\tfrac{1}{1+\alpha}\mLambda_2$ is diagonal, and the diagonal entries are of the form $\frac{\lambda_{\ell}}{\lambda_{\ell}^2+\alpha} - \frac{\lambda_{\ell}}{1+\alpha} = \frac{\lambda_{\ell}(1-\lambda_{\ell}^2)}{(1+\alpha)(\lambda_{\ell}^2+\alpha)}$ where $\lambda_{\ell}$ satisfies either $0 \le \lambda_{\ell} \le \alpha(1+\alpha)\eps$ or $1-\frac{1}{3}\eps \le \lambda_{\ell} \le 1$. If $\lambda_{\ell} \le \alpha(1+\alpha)\eps$, then we have: $$0 \le \dfrac{\lambda_{\ell}(1-\lambda_{\ell}^2)}{(1+\alpha)(\lambda_{\ell}^2+\alpha)} \le \dfrac{\alpha(1+\alpha)\eps \cdot 1}{(1+\alpha)\alpha} = \eps.$$
If $1-\tfrac{1}{3}\eps \le \lambda_{\ell} \le 1$, then since $0 < \eps < \tfrac{1}{2}$, we also have $\lambda_{\ell} \ge 1-\tfrac{1}{3}\eps \ge \tfrac{5}{6}$, and thus: $$0 \le \dfrac{\lambda_{\ell}(1-\lambda_{\ell}^2)}{(1+\alpha)(\lambda_{\ell}^2+\alpha)} = \dfrac{\lambda_{\ell}(1+\lambda_{\ell})(1-\lambda_{\ell})}{(1+\alpha)(\lambda_{\ell}^2+\alpha)}\le \dfrac{1 \cdot 2 \cdot \tfrac{1}{3}\eps}{1 \cdot (\tfrac{5}{6})^2} = \dfrac{24}{25}\eps \le \eps.$$
In either case, $0 \le \frac{\lambda_{\ell}(1-\lambda_{\ell}^2)}{(1+\alpha)(\lambda_{\ell}^2+\alpha)} \le \eps$. Hence, $\|(\mLambda_2^2+\alpha\mId)^{-1}\mLambda_2-\tfrac{1}{1+\alpha}\mLambda_2\| \le \eps$, and thus $\|\mE\| \le \eps$.
\end{proof}

\section{Simulations}
\label{sec:simulations}
We close in this section by presenting several numerical simulations comparing our fast, approximate algorithms to the exact versions. All of our simulations were performed via a MATLAB software package that we have made available for download at http://mdav.ece.gatech.edu/software/. This software package contains all of the code necessary to reproduce the experiments and figures described in this paper.

\subsection{Fast projection onto the span of $\mS_K$}
To test our fast factorization of $\mS_K\mS_K^*$ and our fast projection method, we fix the half-bandwidth $W = \tfrac{1}{4}$ and vary the signal length $N$ over several values between $2^8$ and $2^{20}$. For each value of $N$ we randomly generate several length-$N$ vectors and project each one onto the span of the first $K = \operatorname{round}(2NW)$ elements of the Slepian basis using the fast factorization $\mT_1\mT_2^*$ and the fast projection matrix $\mB_{N,W} + \mU_1\mU_2^*$ for tolerances of $\eps = 10^{-3}$, $10^{-6}$, $10^{-9}$, and $10^{-12}$. The prolate matrix, $\mB_{N,W}$, is applied to the length $N$ vectors via an FFT whose length is the smallest power of $2$ that is at least $2N$. For values of $N \le 12288$, we also projected each vector onto the span of the first $K$ elements of the Slepian basis using the exact projection matrix $\mS_K\mS_K^*$. The exact projection could not be tested for values of $N > 12288$ due to computational limitations. A plot of the average time needed to project a vector onto the span of the first $K = \operatorname{round}(2NW)$ elements of the Slepian basis using the exact projection matrix $\mS_K\mS_K^*$ and the fast factorization $\mT_1\mT_2^*$ is shown in the top left in Figure~\ref{fig:SlepianProjectionTime}. A similar plot comparing the exact projection $\mS_K\mS_K^*$ and the fast projection $\mB_{N,W}+\mU_1\mU_2^*$ is shown in the top right in Figure~\ref{fig:SlepianProjectionTime}. As can be seen in the figures the time required by the exact projection grows quadratically with $N$, while the time required by the fast factorization as well as the fast projection grows roughly linearly in $N$.

For the exact projection, all of the first $K = \operatorname{round}(2NW)$ elements of the Slepian basis must be precomputed. For the fast factorization, the low rank matrices $\mL_1, \mL_2$ (from Theorem~\ref{thm:prolateFFTLR}) and the Slepian basis elements $\vs_{N,W}^{(\ell)}$ for which $\eps < \lambda_{N,W}^{(\ell)} < 1-\eps$ are precomputed. For the fast projection, the FFT of the sinc kernel, as well as the Slepian basis elements $\vs_{N,W}^{(\ell)}$ for which $\eps < \lambda_{N,W}^{(\ell)} < 1-\eps$ are precomputed.  A plot of the average precomputation time needed for both the exact projection $\mS_K\mS_K^*$ as well as the fast factorization $\mT_1\mT_2^*$ is shown in the top left in Figure~\ref{fig:SlepianProjectionSetupTime}. A similar plot comparing the exact projection $\mS_K\mS_K^*$ and the fast projection $\mB_{N,W}+\mU_1\mU_2^*$ is shown in the top right in Figure~\ref{fig:SlepianProjectionSetupTime}. As can be seen in the figures the precomputation time required by the exact projection grows roughly quadratically with $N$, while the precomputation time required by the fast factorization as well as the fast projection grows just faster than linearly in $N$.

This experiment was repeated with $W = \tfrac{1}{16}$ and $W = \tfrac{1}{64}$ (instead of $W = \tfrac{1}{4}$). The results for $W = \tfrac{1}{16}$ and $W = \tfrac{1}{64}$ are shown in the middle and bottom, respectively, of Figures~\ref{fig:SlepianProjectionTime} and ~\ref{fig:SlepianProjectionSetupTime}. The exact projection onto the first $K \approx 2NW$ elements of the Slepian basis takes $O(NK) = O(2WN^2)$ operations, whereas both our fast factorization and fast projection algorithms take $O(N \log N \log \tfrac{1}{\eps})$ operations. The smaller $W$ gets, the larger $N$ needs to be for our fast methods to be faster than the exact projection via matrix multiplication. If $W \lesssim \tfrac{1}{N}\log N \log \tfrac{1}{\eps}$, then our fast methods lose their computational advantage over the exact projection. However, in this case the exact projection is fast enough to not require a fast approximate algorithm.  

\begin{figure}%
   \centering
  \includegraphics[scale = 0.38]{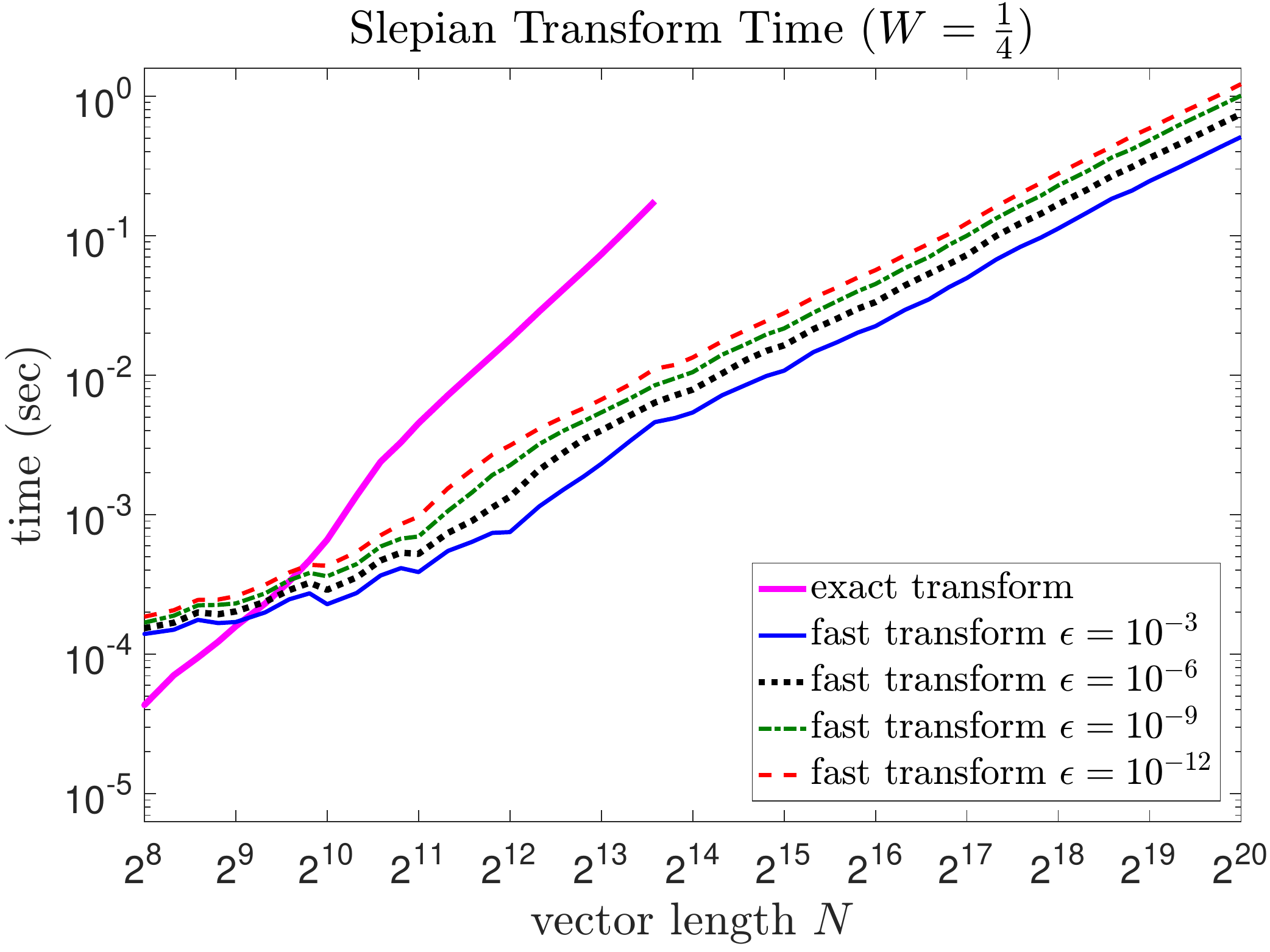} \includegraphics[scale = 0.38]{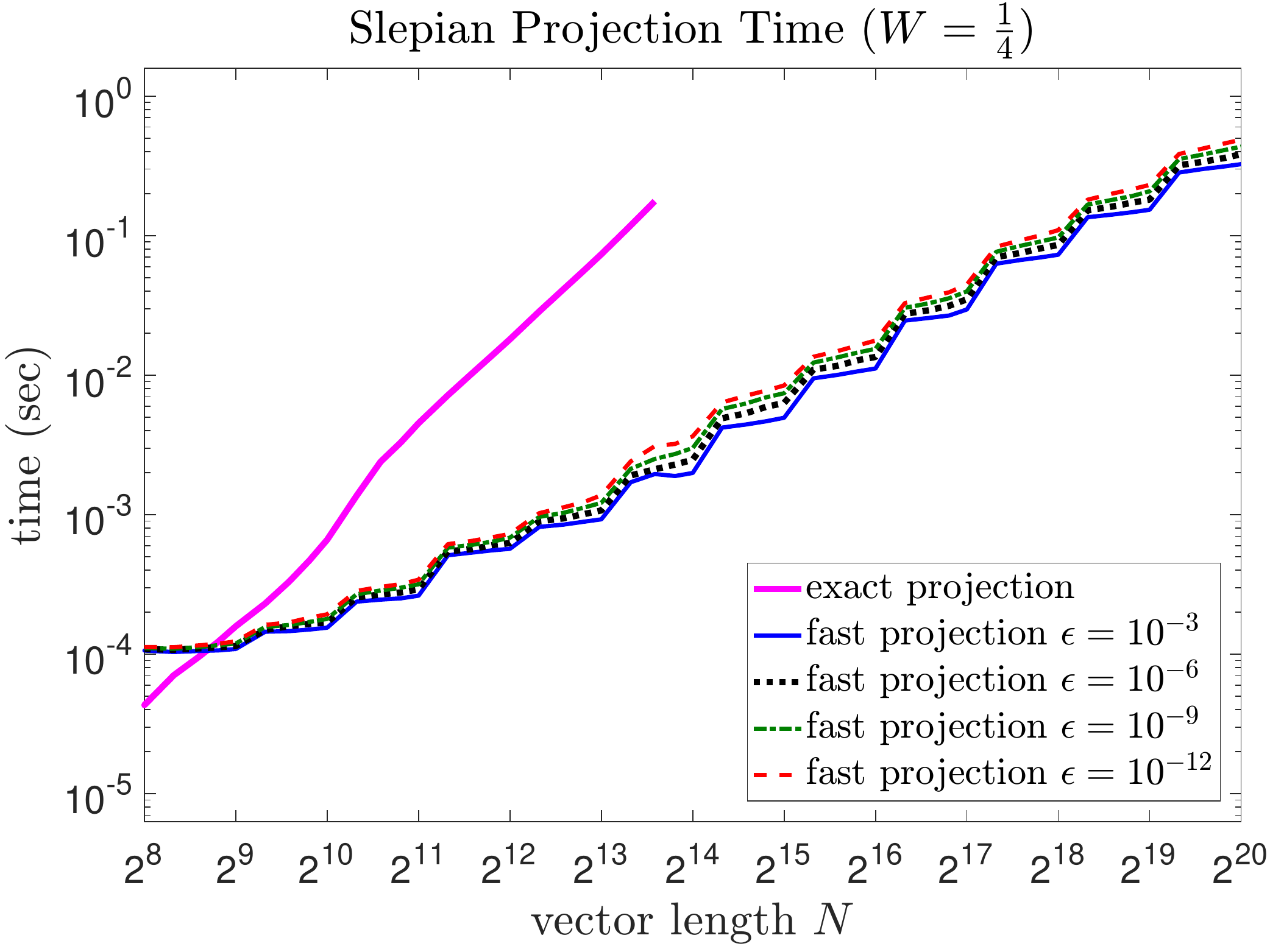}

  \includegraphics[scale = 0.38]{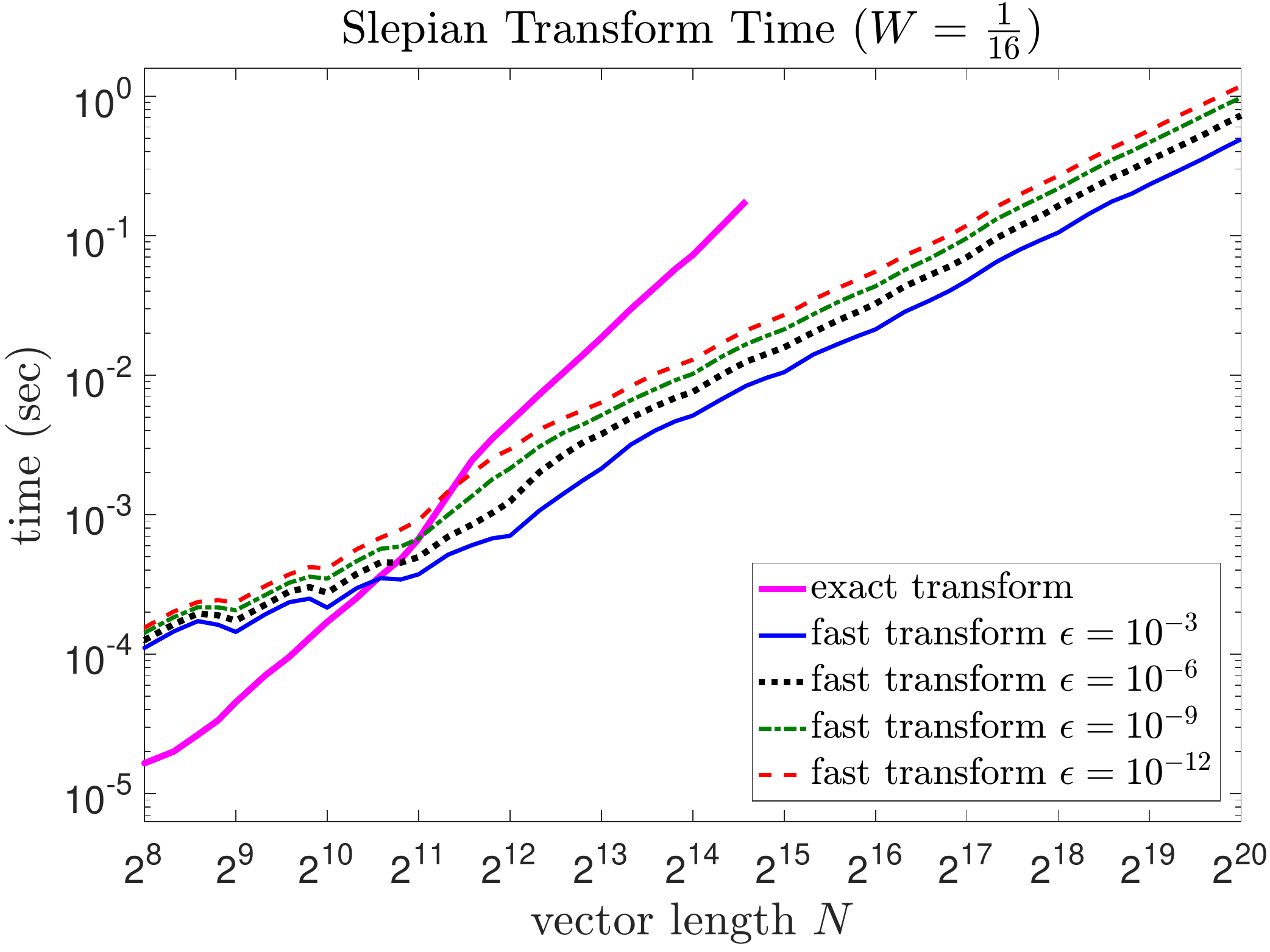} \includegraphics[scale = 0.38]{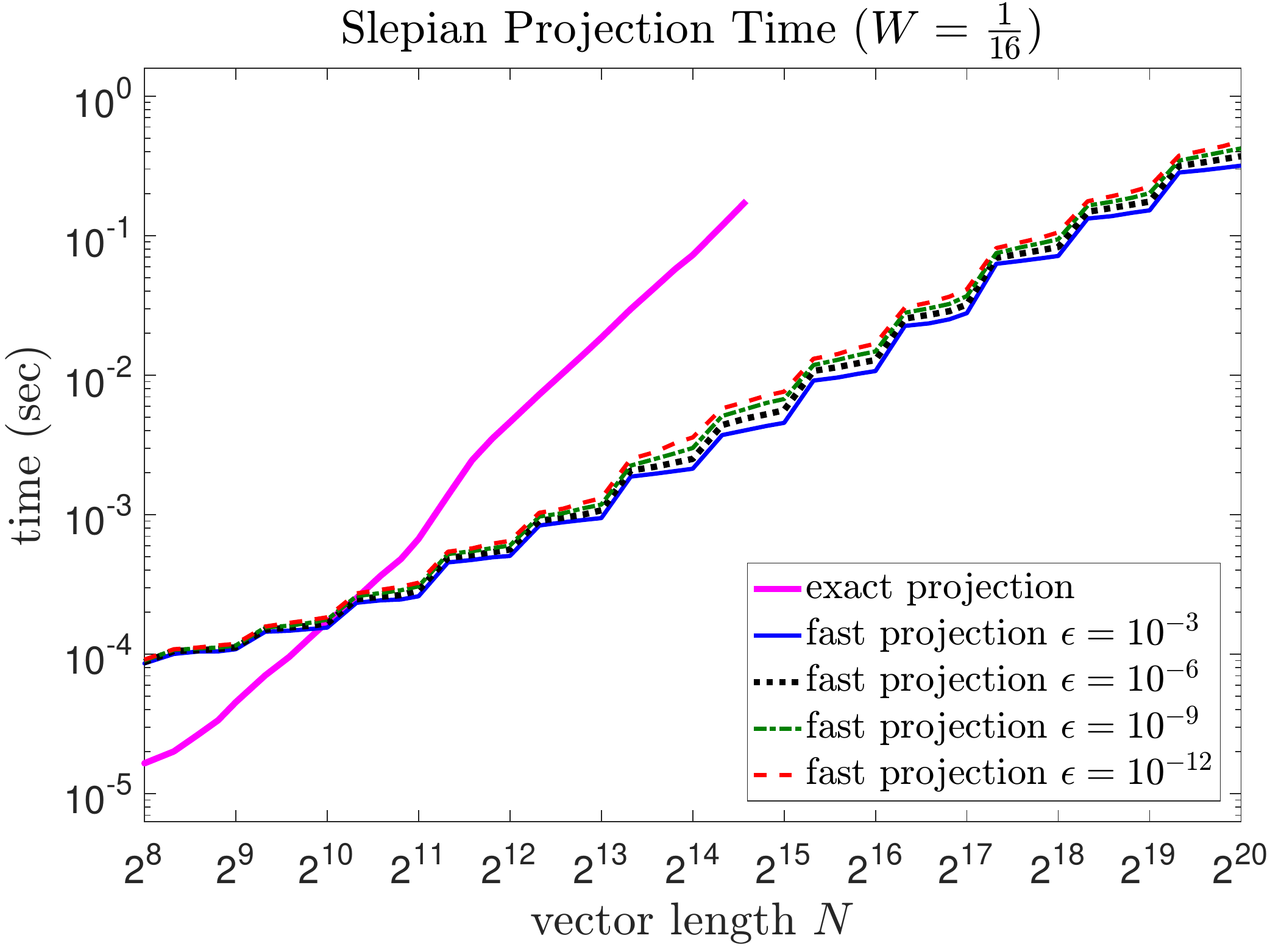}

  \includegraphics[scale = 0.38]{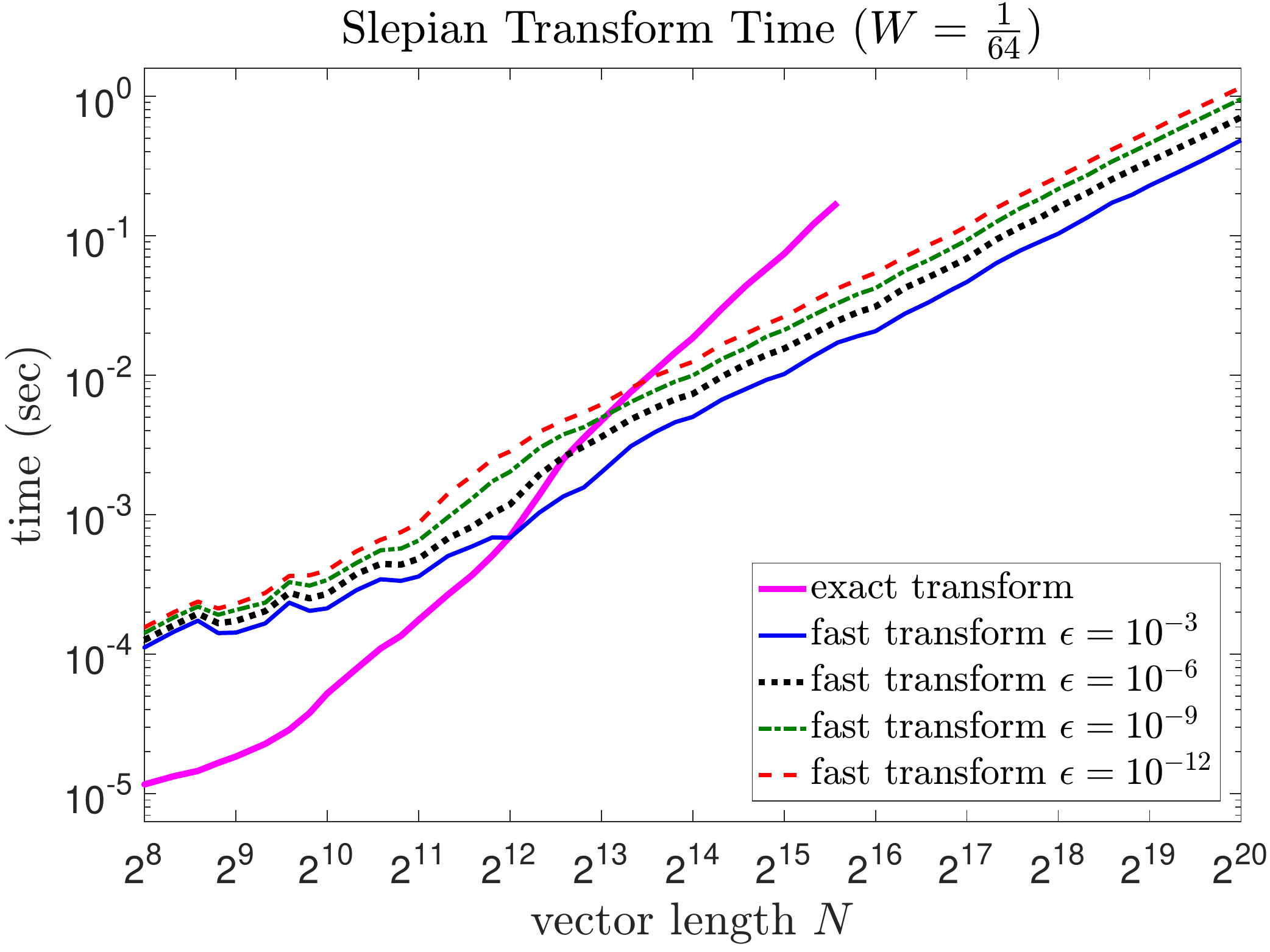} \includegraphics[scale = 0.38]{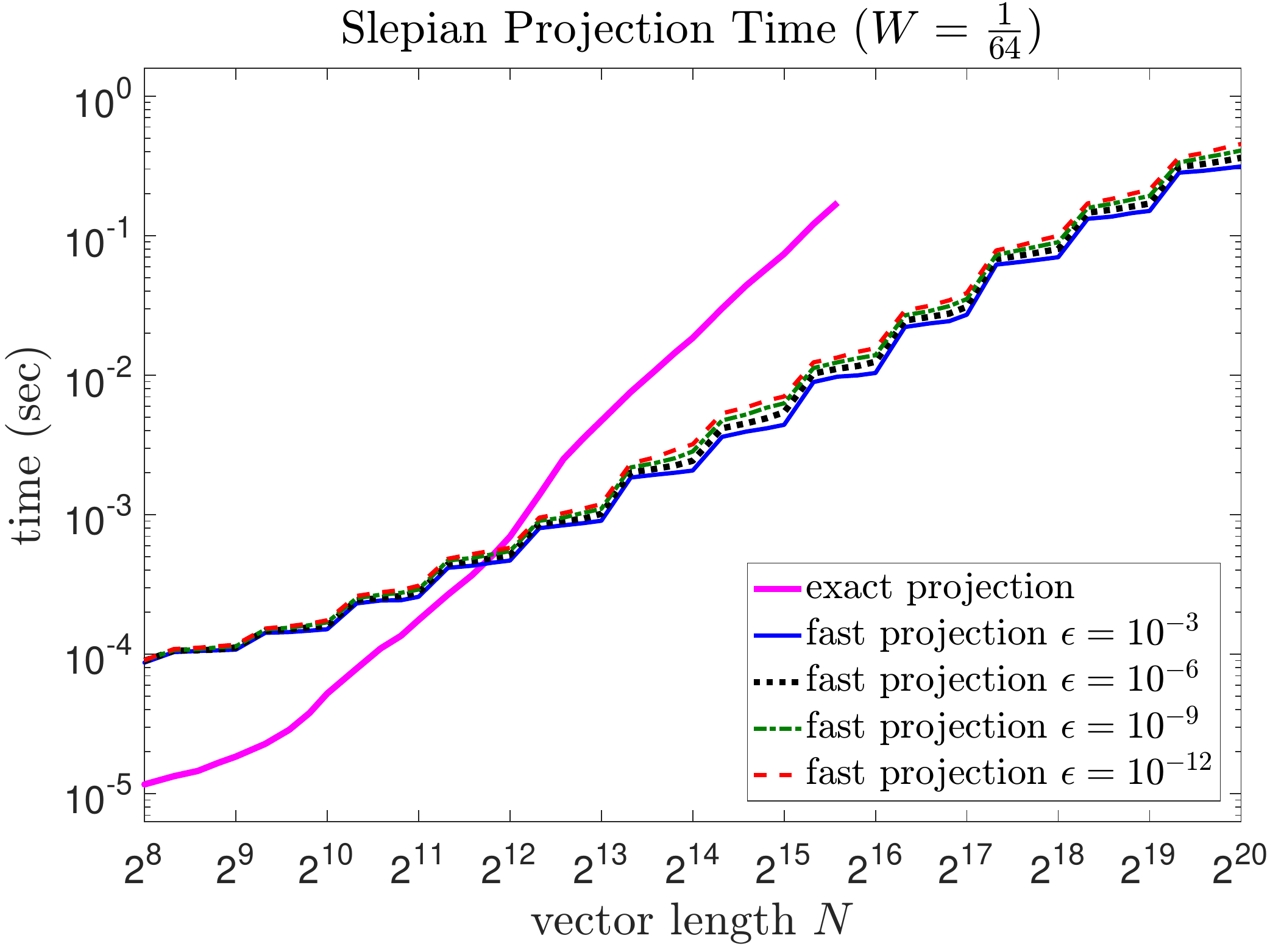}

   \caption{\small \sl (Left) Plots of the average time needed to project a vector onto the first $\operatorname{round}(2NW)$ Slepian basis elements using the exact projection $\mS_K\mS_K^*$ and using the fast factorization $\mT_1\mT_2^*$. (Right) Plots of the average time needed to project a vector onto the first $\operatorname{round}(2NW)$ Slepian basis elements using the exact projection $\mS_K\mS_K^*$ and using the fast projection $\mB_{N,W}+\mU_1\mU_2^*$.  \label{fig:SlepianProjectionTime}}
\end{figure}

\begin{figure}%
   \centering
  \includegraphics[scale = 0.38]{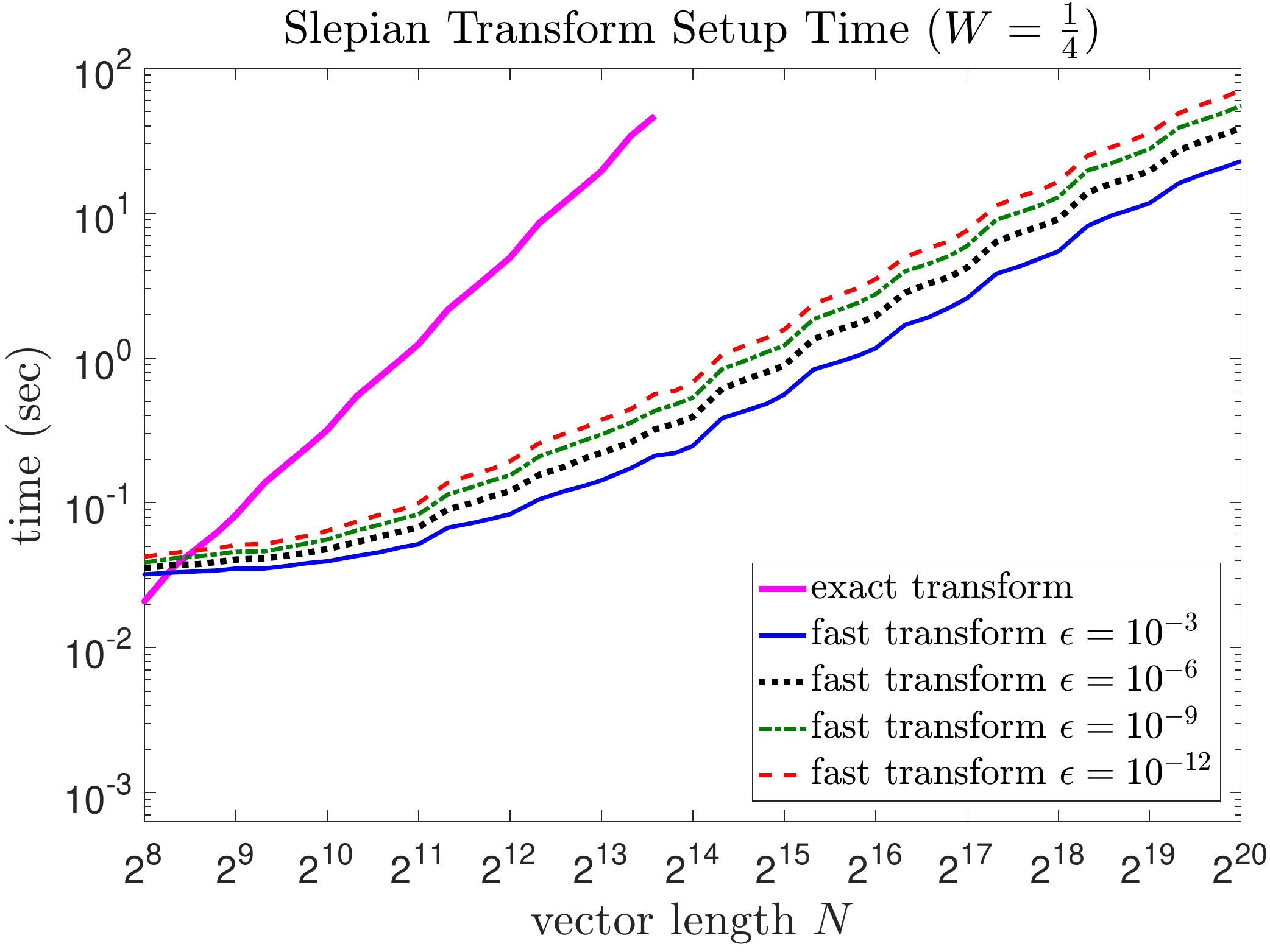} \includegraphics[scale = 0.38]{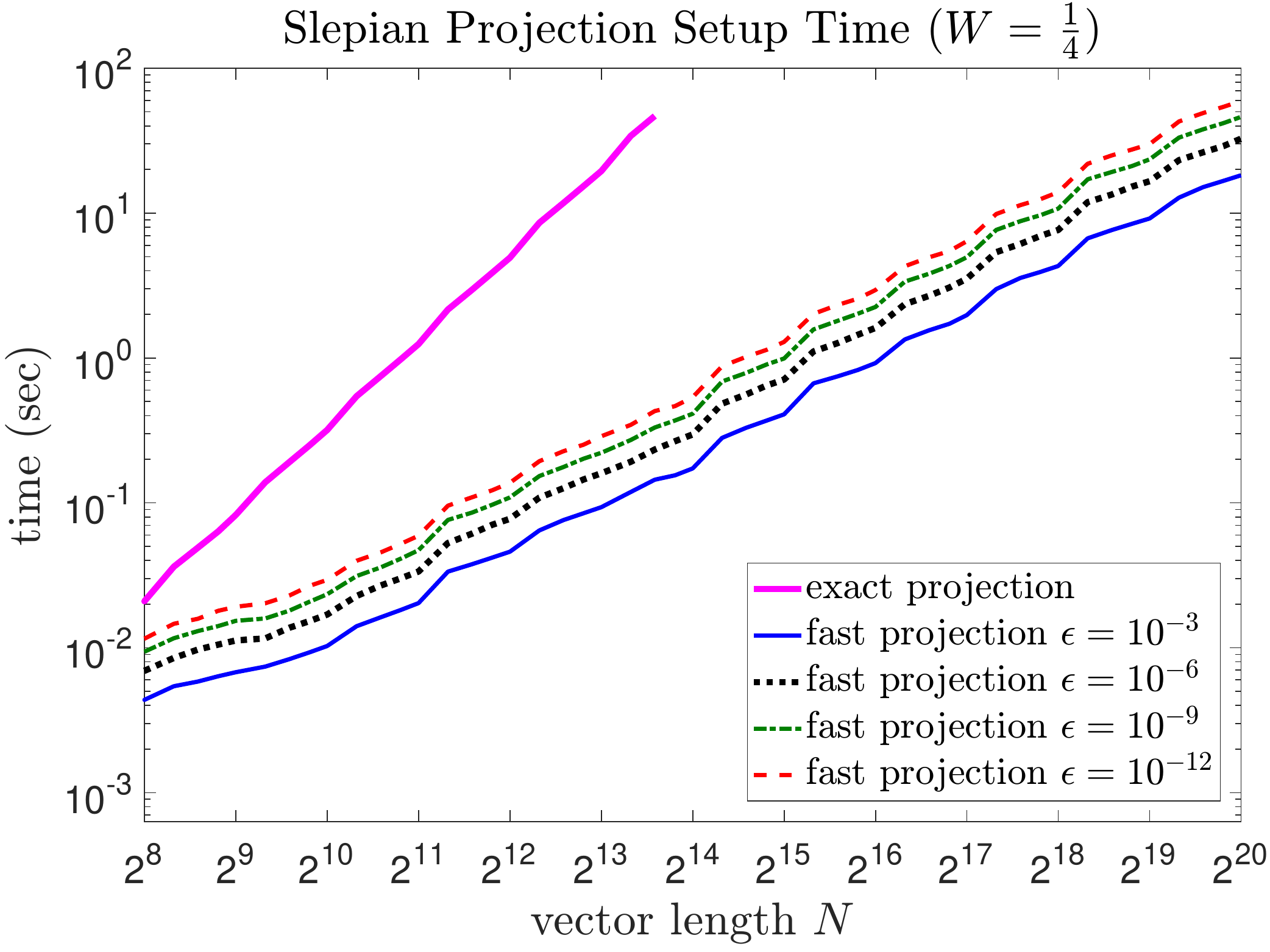}

  \includegraphics[scale = 0.38]{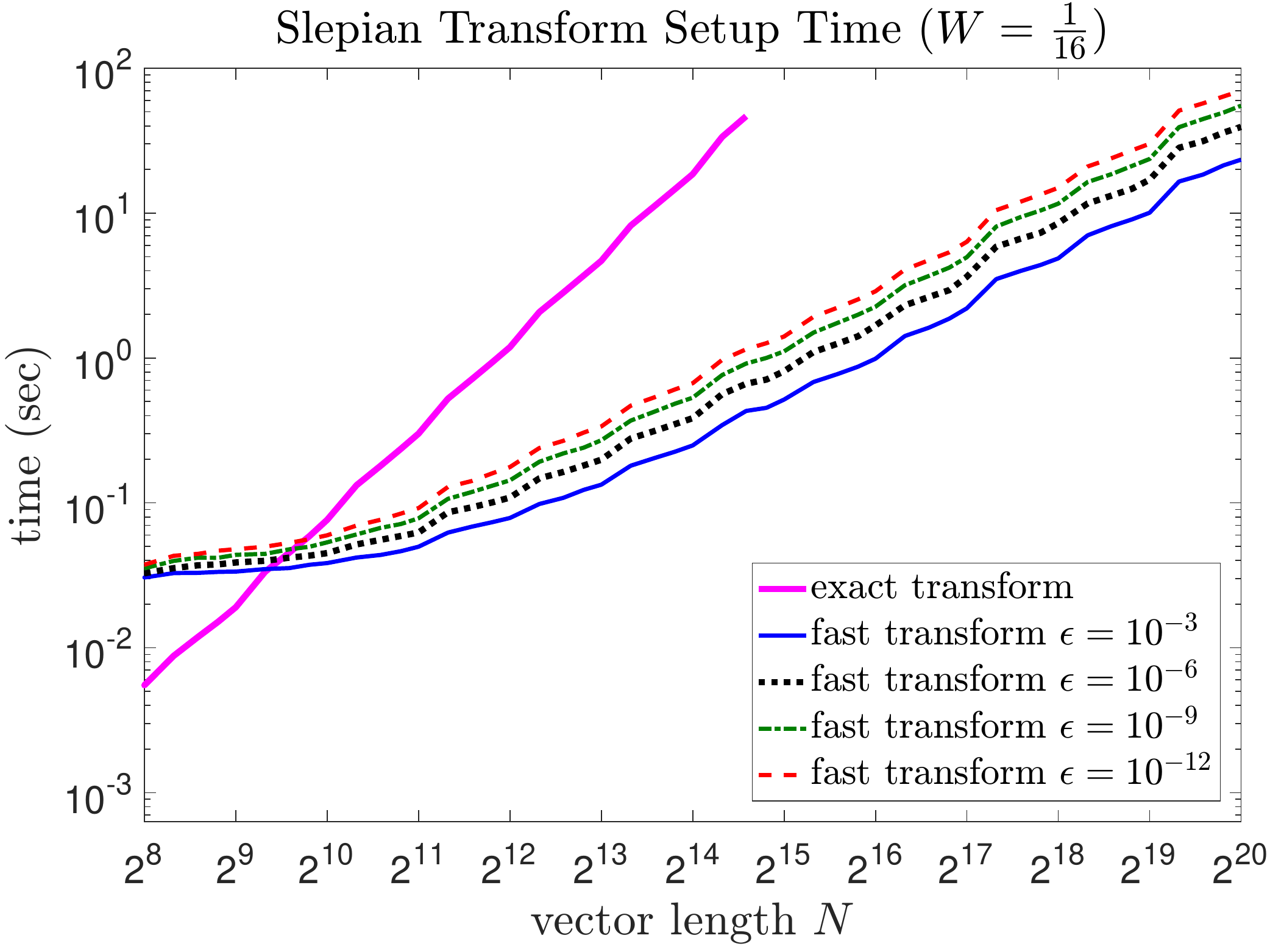} \includegraphics[scale = 0.38]{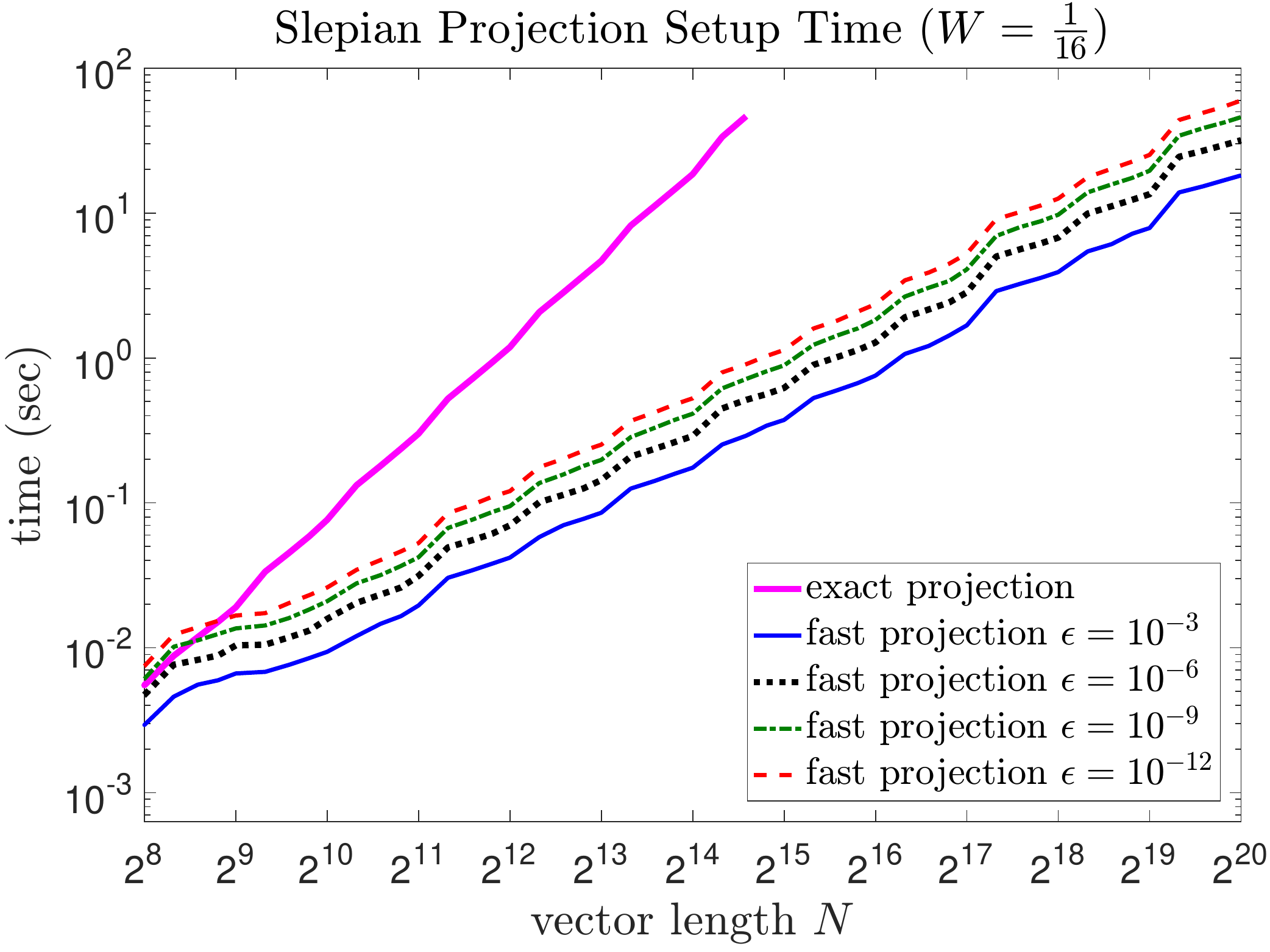}

  \includegraphics[scale = 0.38]{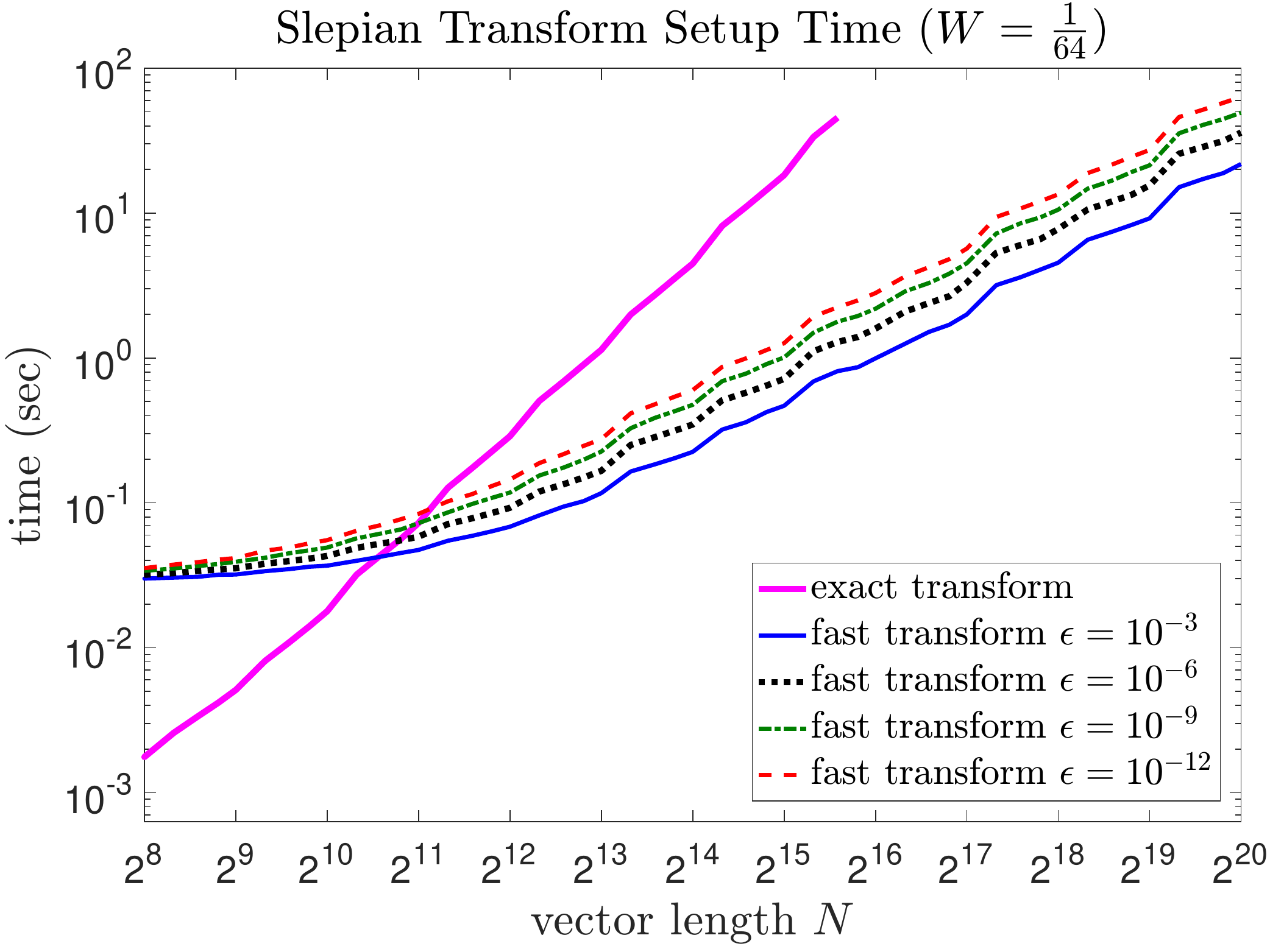} \includegraphics[scale = 0.38]{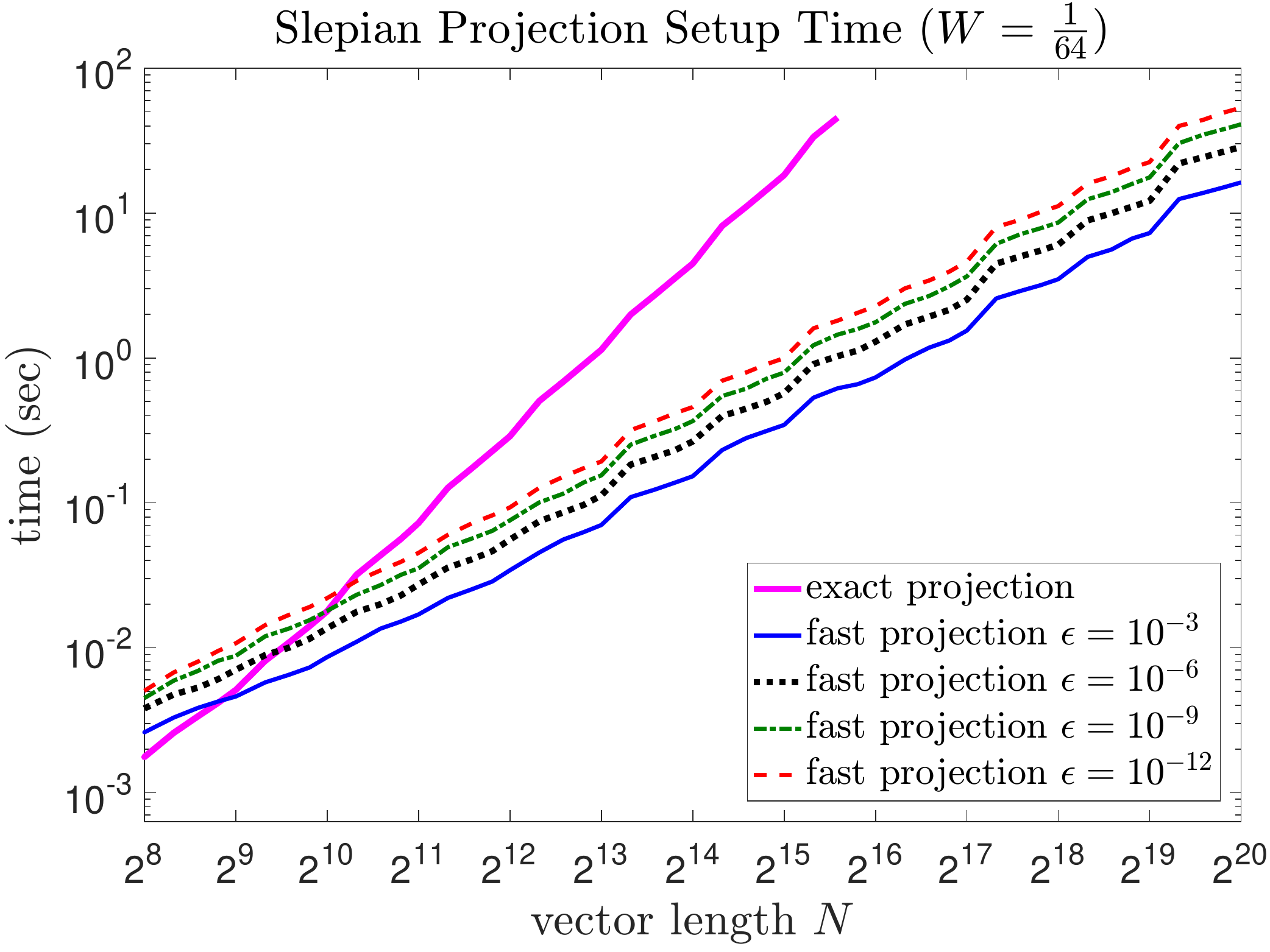}

   \caption{\small \sl (Left) Plots of the average precomputation time for the exact projection $\mS_K\mS_K^*$ and the fast factorization $\mT_1\mT_2^*$. (Right) Plots of the average precomputation time for the exact projection $\mS_K\mS_K^*$ and the fast projection $\mB_{N,W}+\mU_1\mU_2^*$. \label{fig:SlepianProjectionSetupTime}}
\end{figure}

\subsection{Solving least-squares systems involving $\mB_{N,W}$}
\label{sec:leastsqsystems}
We demonstrate the effectiveness of our fast prolate pseudoinverse method (Corollary~\ref{cor:prolatePseudoinverseLR}) and our fast prolate Tikhonov regularization method (Corollary~\ref{cor:prolateTikhonovLR}) on an instance of the Fourier extension problem, as described in Section~\ref{ssec:apps}. 

To choose an appropriate function $f$, we note that if $f$ is continuous and $f(-1) = f(1)$, then the Fourier sum approximations will not suffer from Gibbs phenomenon, and so, there is no need to compute a Fourier extension sum approximation for $f$. Also, if $f$ is smooth on $[-1,1]$ but $f(-1) \neq f(1)$, then the Fourier sum approximations will suffer from Gibbs phenomenon, but the Fourier extension series coefficients will decay exponentially fast. Hence, relatively few Fourier extension series coefficients will be needed to accurately approximate $f$, which makes the least squares problem of solving for these coefficients small enough for our fast methods to not be useful. However, in the case where $f$ is continuous but not smooth on $[-1,1]$ and $f(-1) \neq f(1)$, the Fourier series will suffer from Gibbs phenomenon, and the Fourier extension series coefficients will decay faster than the Fourier series coefficients, but not exponentially fast. So in this case, the number of Fourier extension series coefficients required to accurately approximate $f$ is not trivially small, but still less than the number of Fourier series coefficients required to accurately approximate $f$. Hence, computing a Fourier extension sum approximation to $f$ is useful and requires our fast methods.

We construct such a function $f : [-1,1] \to \R$ in the form $f(t) = a_0t + \sum_{\ell = 1}^{L}a_{\ell}\exp(-\tfrac{|t-\mu_{\ell}|}{\sigma_{\ell}})$ where $a_0 = 5$, $L = 500$, and $a_{\ell}$, $\mu_{\ell}$, and $\sigma_{\ell}$ are chosen in a random manner. A plot of $f(t)$ over $t \in [-1,1]$ is shown on the left in Figure~\ref{fig:FFE_Function}. Also on the right in Figure~\ref{fig:FFE_Function}, we show an example of a Fourier sum approximation and a Fourier extension approximation, both with $401$ terms. Notice that the Fourier sum approximation suffers from Gibbs phenomenon near the endpoints of the interval $[-1,1]$, while the Fourier extension approximation does not exhibit such oscillations near the endpoints of $[-1,1]$. 

\begin{figure}%
   \centering
\includegraphics[scale = 0.38]{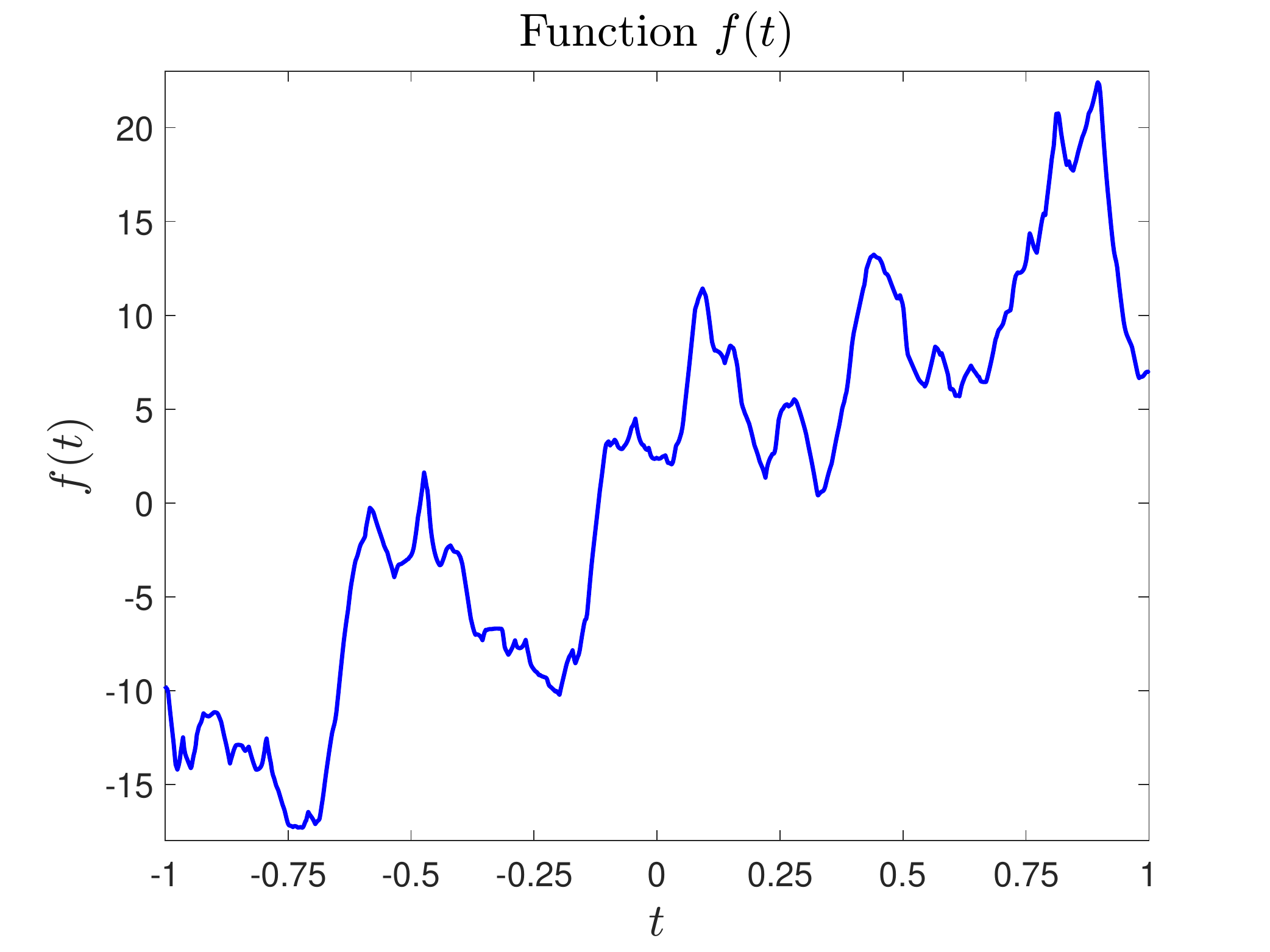} \includegraphics[scale = 0.38]{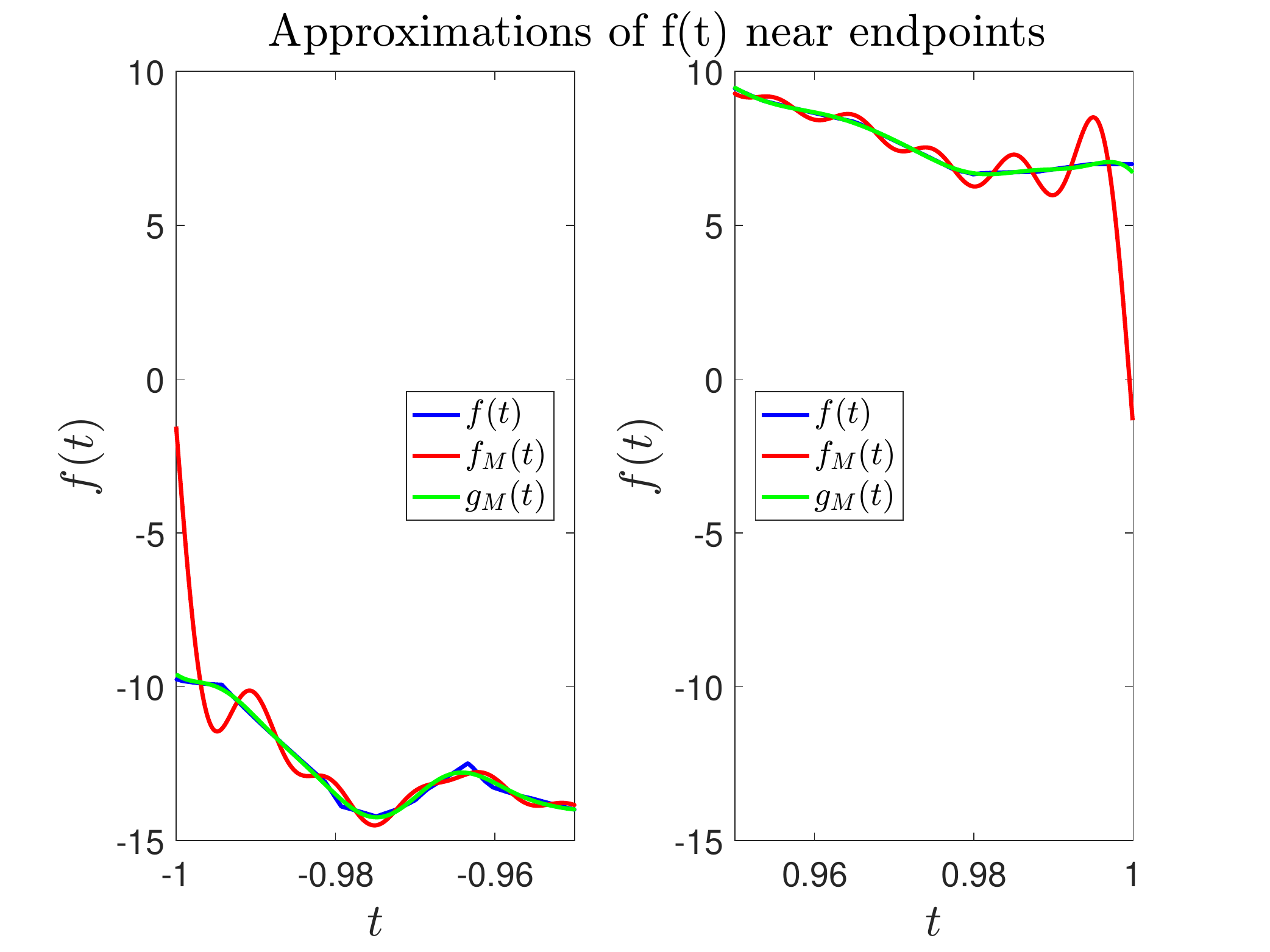}
   \caption{\small \sl (Left) A plot of the function used in the experiments described in Section~\ref{sec:leastsqsystems}. (Right) Plots of the function, the Fourier sum approximation to $f(t)$ using $401$ terms, and the Fourier extension approximation to $f(t)$ using $401$ terms. Note that the Fourier sum approximation suffers from Gibbs phenomenon oscillations while the Fourier extension sum does not. \label{fig:FFE_Function}}
\end{figure}
\vspace{0.1in}

For several positive integers $M$ between $1$ and $2560$, we compute three approximations to $f(t)$:
\begin{enumerate}[(i)]
\item The $2M+1$ term truncated Fourier series of $f(t)$, i.e., $$f_{M}(t) = \dfrac{1}{\sqrt{2}}\displaystyle\sum_{m = -M}^{M}\widehat{f}_me^{j\pi mt}, \ \text{where} \ \widehat{f}_m = \dfrac{1}{\sqrt{2}}\int_{-1}^{1}f(t)e^{j\pi mt}\,dt.$$

\item The $2M+1$ term Fourier extension of $f(t)$ to the interval $[-T,T]$, i.e., $$g_{M}(t) = \dfrac{1}{\sqrt{2T}}\displaystyle\sum_{m = -M}^{M}\widehat{g}_me^{j\pi mt/T}, \ \text{where} \ \widehat{y}_m = \dfrac{1}{\sqrt{2T}}\int_{-1}^{1}f(t)e^{j\pi mt/T}\,dt \ \text{and} \ \widehat{g} = \mB_{(2M+1),\tfrac{1}{2T}}^{\dagger}\widehat{y}.$$ Here, we pick $T = 1.5$, and we let $\mB_{(2M+1),\tfrac{1}{2T}}^{\dagger}$ be the truncated pseudoinverse of $\mB_{(2M+1),\tfrac{1}{2T}}$ which zeros out eigenvalues smaller than $10^{-4}$.

\item The $2M+1$ term Fourier extension of $f(t)$ to the interval $[-T,T]$ (as described above), except we use the fast prolate pseudoinverse method (Corollary~\ref{cor:prolatePseudoinverseLR}) with tolerance $\eps = 10^{-5}$ instead of the exact truncated pseudoinverse.
\end{enumerate}

The integrals used in computing the coefficients are approximated using an FFT of length $2^{13+q}$ where $q = \left\lfloor\log_2 M\right\rfloor$. By increasing the FFT length with $M$, we ensure that the coefficients are sufficiently approximated, while also ensuring that the time needed to compute the FFT does not dominate the time needed to solve the system $\mB_{(2M+1),\tfrac{1}{2T}}\widehat{g} = \widehat{y}$. Given an approximation $\hat{f}(t)$ to $f(t)$, we quantify the performance via the relative root-mean-square (RMS) error:

\vspace{-0.15in}
 $$\frac{\|f-\hat{f}\|_{L^2[-1,1]}}{\|f\|_{L^2[-1,1]}}$$
\vspace{-0.10 in}

A plot of the relative RMS error versus $M$ for each of the three approximations to $f(t)$ is shown on the left in Figure~\ref{fig:FFE_Pinv}. For values of $M$ at least $448$, the Fourier extension $g_M(t)$ (computed with either the exact or the fast pseudoinverse) yielded a relative RMS error at least $10$ times lower than that for the truncated Fourier series $f_M(t)$. Using the exact pseudoinverse instead of the fast pseudoinverse does not yield a noticable improvement in the approximation error. A plot of the average time needed to compute the approximation coefficients versus $M$ is shown on the right in Figure~\ref{fig:FFE_Pinv}. For large $M$, computing the Fourier extension coefficients using the fast prolate pseudoinverse is significantly faster than computing the Fourier extension coefficients  using the fast prolate pseudoinverse. Also, computing the Fourier extension coefficients using the fast prolate pseudoinverse takes only around twice the time required for computing the Fourier series coefficients.
\begin{figure}%
   \centering
  \includegraphics[scale = 0.38]{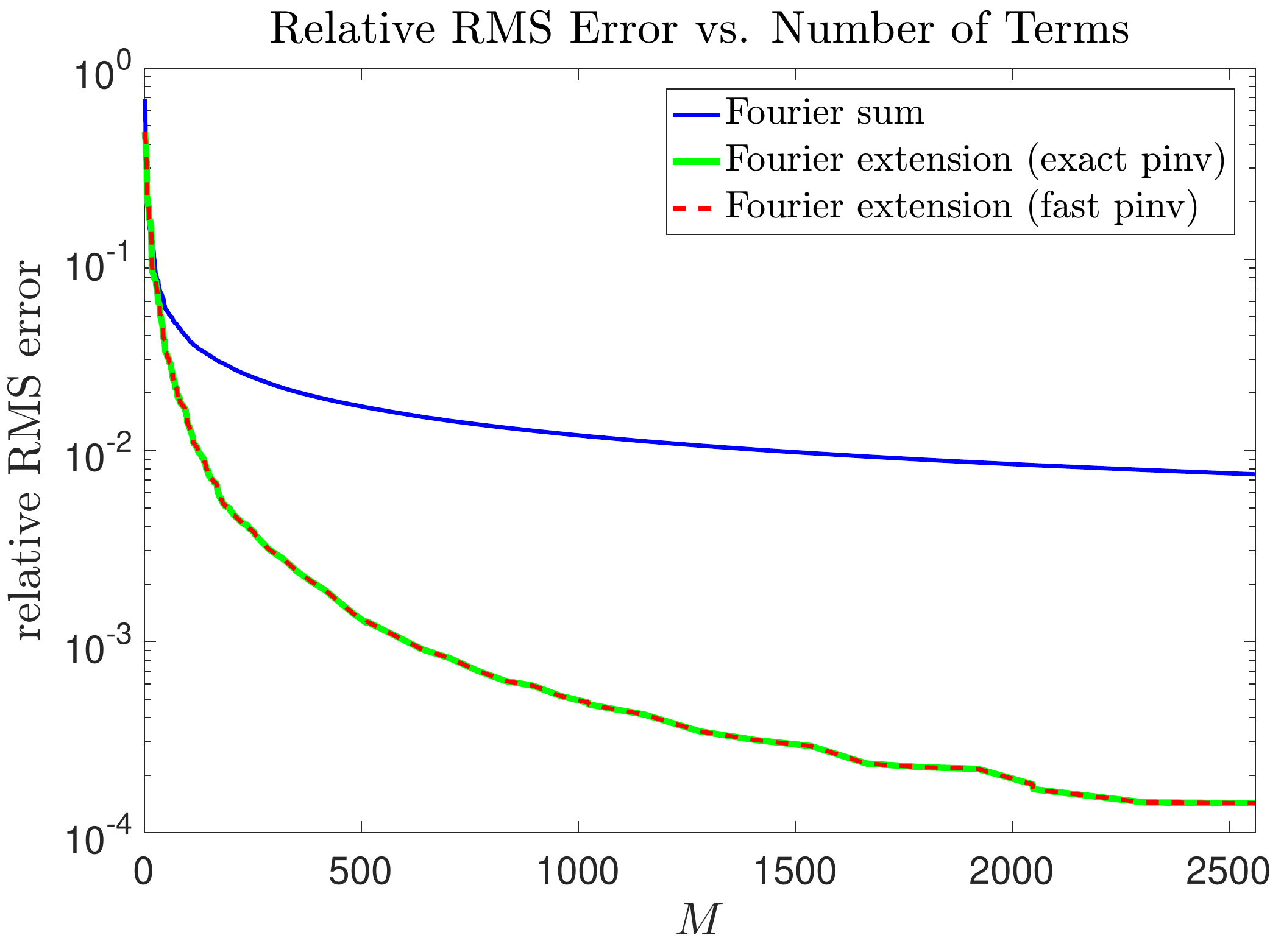} \includegraphics[scale = 0.38]{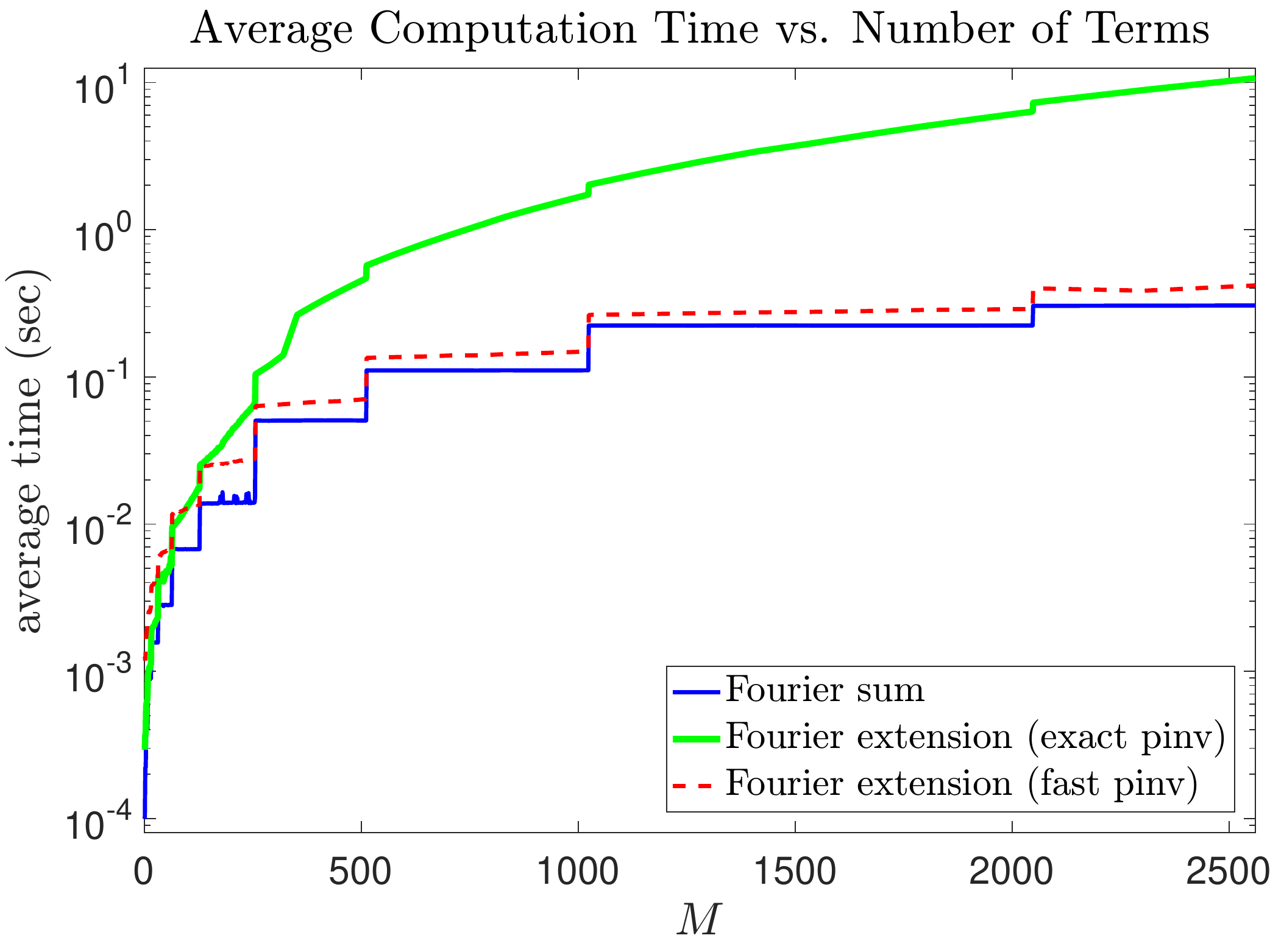}
   \caption{\small \sl A comparison of the relative RMS error (left) and the computation time required (right) for the $2M+1$ term truncated Fourier series as well as the $2M+1$ term Fourier extension using both the exact and fast pseudoinverse methods. Note that the exact and fast methods are virtually indistinguishable in terms of relative RMS error.\label{fig:FFE_Pinv}}
\end{figure}

We repeated this experiment, except using Tikhonov regularization to solve the system $\mB_{(2M+1),\tfrac{1}{2T}}\widehat{g} = \widehat{y}$ instead of the truncated pseudoinverse. We tested both the exact Tikhonov regularization procedure $\widehat{g} = (\mB_{(2M+1),\tfrac{1}{2T}}^2+\alpha\mId)^{-1}\mB_{(2M+1),\tfrac{1}{2T}}\widehat{y}$ (for $\alpha = 10^{-8}$) as well as the fast Tikhonov regularization method (Corollary~\ref{cor:prolateTikhonovLR}) with a tolerance of $\eps = 10^{-5}$. The results, which are similar to those for the pseudoinverse case, are shown in Figure~\ref{fig:FFE_Tik}.
\begin{figure}%
   \centering
  \includegraphics[scale = 0.38]{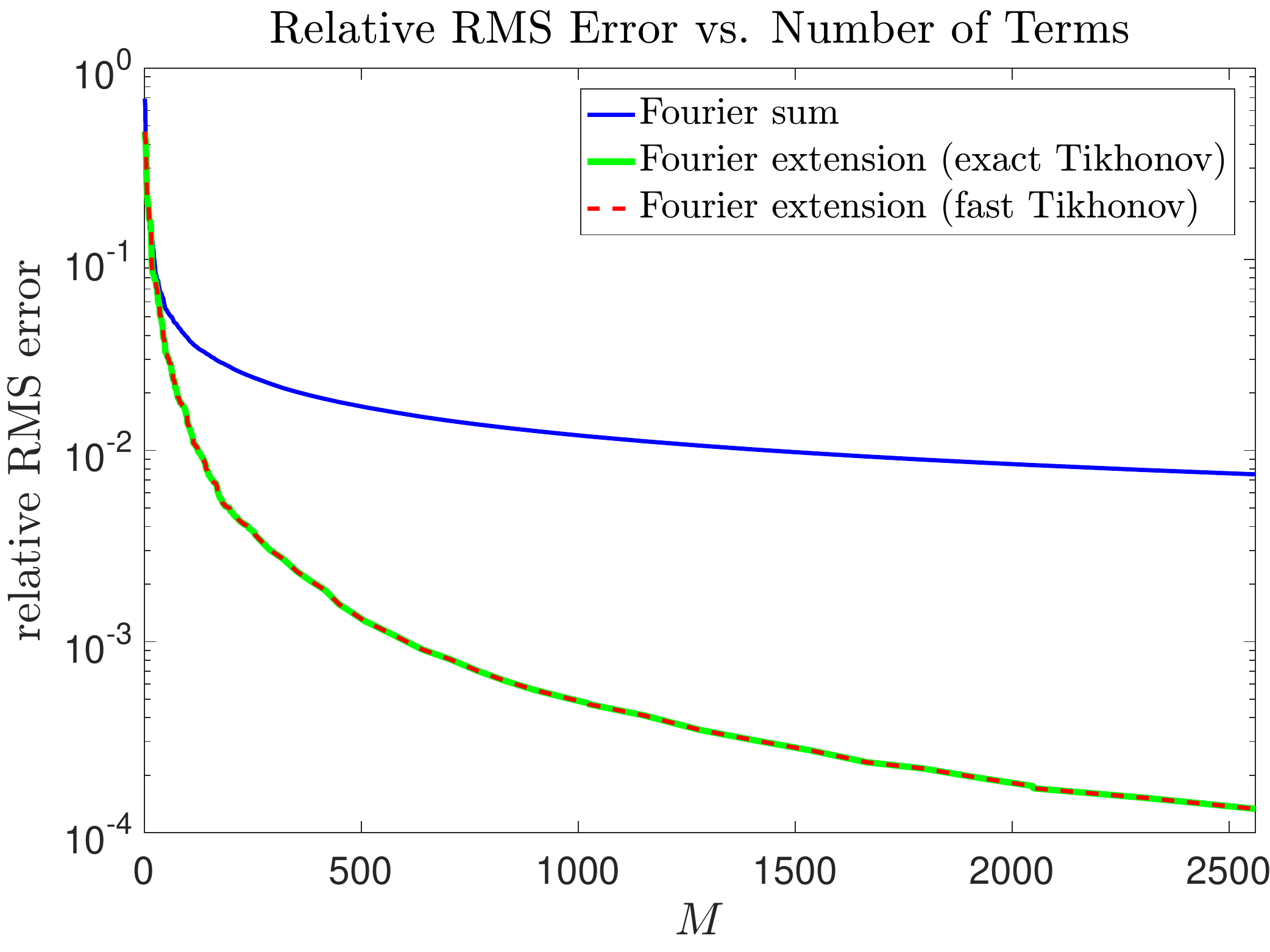} \includegraphics[scale = 0.38]{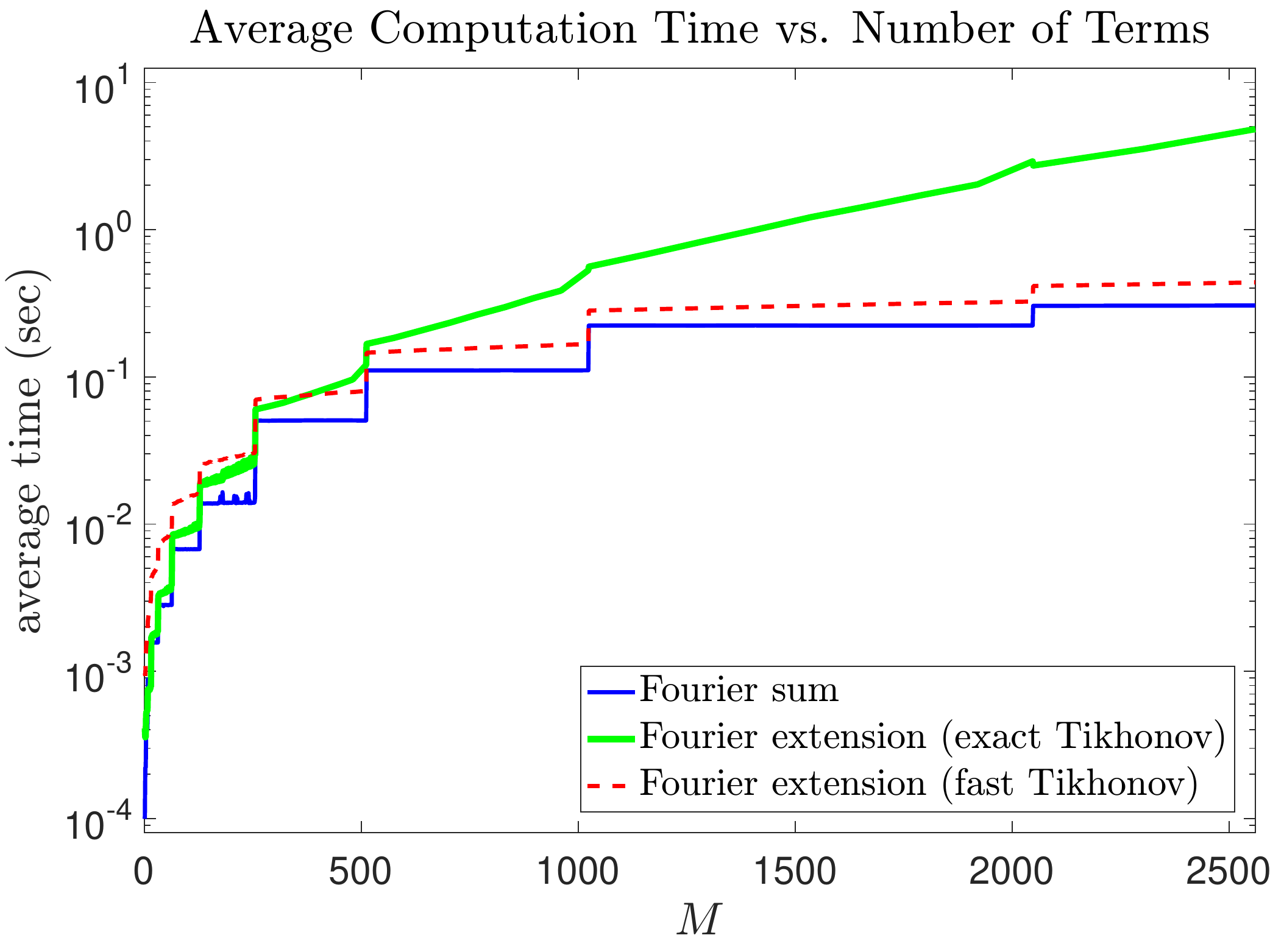}
   \caption{\small \sl A comparison of the relative RMS error (left) and the computation time required (right) for the $2M+1$ term truncated Fourier series as well as the $2M+1$ term Fourier extension using both the exact and fast Tikhonov regularization methods. Note that the exact and fast methods are virtually indistinguishable in terms of relative RMS error.\label{fig:FFE_Tik}}
\end{figure}

%--------------------------------------------------------------------------

\appendix
\section{Proof of Lemma~\ref{lem:Lyapunov}}

Iterative methods for efficiently computing a low-rank approximation to the solution of a Lyapunov system have been well-studied \cite{lu91so,wachspress13ad}.  The CF-ADI algorithm presented in \cite{li02lo} constructs a factor $\mZ$ by concatenating a series of $r$ $N\times M$ matrices, $\mZ = \begin{bmatrix} \mZ_1 & \mZ_2 & \cdots & \mZ_r \end{bmatrix}$, where
\begin{align*}
	\mZ_1 &= \sqrt{2p_1}(\mA+p_1\mId)^{-1}\mB \\
	\mZ_k &= \sqrt{\frac{p_k}{p_{k-1}}}\left(\mId - (p_k+p_{k-1})(\mA+p_k\mId)^{-1}\right)\mZ_{k-1}, \quad k=2,\ldots,r,
\end{align*}
for some choice of positive real numbers $p_1,\ldots,p_r$.  It is shown in \cite{li02lo} that the matrix $\mZ\mZ^*$ produced by this iteration is equivalent to the matrix produced by the ADI iteration given in \cite{lu91so}, and thus, $\mZ\mZ^*$ satisfies
\[
\mX-\mZ\mZ^* = \phi(\mA)\mX\phi(\mA)^* \ \text{where} \ \phi(x) = \prod_{j = 1}^{r}\dfrac{x-p_j}{x+p_j}.
\]
(This is shown in \cite{lu91so} by using induction on $r$.) Therefore, the error $\|\mX-\mZ\mZ^*\|$ satisfies
\[
\|\mX-\mZ\mZ^*\| \le \|\mX\| \cdot \|\phi(\mA)\|^2 = \|\mX\| \cdot \max_{x \in \text{Spec}(\mA)}|\phi(x)|^2 \le \|\mX\| \cdot \max_{x \in[a,b]}|\phi(x)|^2,
\]
where $a = \lambda_{\text{min}}(\mA)$ and $b = \lambda_{\text{max}}(\mA)$ (so $\kappa = \tfrac{b}{a}$). In \cite{wachspress13ad}, it is shown that for a given interval $[a,b]$ and a number of ADI iterations $r$, there exists a choice of parameters $p_1,\ldots,p_r$ such that $\max_{x \in[a,b]}|\phi(x)|^2 = \alpha$, where $\alpha$ satisfies
\[
\dfrac{I(\sqrt{1-\alpha^2})}{I(\alpha)} = \dfrac{4rI(\kappa^{-1})}{I(\sqrt{1-\kappa^{-2}})},
\]
where $I(\tau)$ is the complete elliptic integral of the first kind, defined by
\[
I(\tau) := \int_{0}^{\pi/2}(1-\tau^2\sin^2\theta)^{-1/2}\,d\theta.
\]
It is shown in \cite{lawden} that the elliptic nome, defined as
\[
q(\tau) := \exp\left[-\pi\dfrac{I(\sqrt{1-\tau^2})}{I(\tau)}\right]
\]
satisfies
\[
\tau^2 = 16q(\tau)\displaystyle\prod_{n = 1}^{\infty}\left(\dfrac{1+q(\tau)^{2n}}{1+q(\tau)^{2n-1}}\right)^8.
\]
For $0 \le \tau \le 1$,  the range of the elliptic nome is $0 \le q(\tau) \le 1$. Hence, the above equation gives us the inequality $\tau^2 \le 16q(\tau)$. By using the definition of the elliptic nome, this inequality becomes
\[
\dfrac{I(\sqrt{1-\tau^2})}{I(\tau)} \le \dfrac{2}{\pi}\log\dfrac{4}{\tau} \ \text{for} \ 0 \le \tau \le 1.
\]
So, by setting the number of iterations as $r =\left\lceil\frac{1}{\pi^2}\log\left(4\kappa\right)\log\left(\tfrac{4}{\delta}\right) \right\rceil$, we have
\[
\dfrac{2}{\pi}\log\dfrac{4}{\alpha} \ge \dfrac{I(\sqrt{1-\alpha^2})}{I(\alpha)} = \dfrac{4rI(\kappa^{-1})}{I(\sqrt{1-\kappa^{-2}})} \ge \dfrac{4 \cdot \tfrac{1}{\pi^2}\log(4\kappa)\log(\tfrac{4}{\delta})}{\tfrac{2}{\pi}\log(4\kappa)} = \dfrac{2}{\pi}\log\dfrac{4}{\delta}.
\]
Hence, $\max_{x \in[a,b]}|\phi(x)|^2 = \alpha \le \delta$, and thus, $\|\mX-\mZ\mZ^*\| \le \delta\|\mX\|$, as desired. \qed

\begin{comment}
They show (building on work in \cite{wachspress13ad}) that for symmetric positive definite $\mA$ and concrete choices of the $\{p_k\}$, \eqref{eq:lylrap} will be satisfied after
\[
	r = \left\lceil\frac{1}{2\pi}\frac{I\left(\sqrt{1-\kappa^{-2}}\right)}{I(\kappa^{-1})}\log(4/\delta)\right\rceil
\]
iterations of the procedure above, where $I(\kappa)$ is the complete elliptic integral
\[
	I(\tau) = \int_0^{\pi/2} \frac{d\theta}{\sqrt{1-\tau^2\sin^2\theta}}.
\]
In \cite[Appendix]{braess05ap}, it is shown that
\[
	\frac{I\left(\sqrt{1-\tau^{2}}\right)}{I(\tau)} ~\leq~ \frac{2}{\pi}\log\left(\frac{4}{\tau}+2\right)
	\quad\text{for all $\tau\in(0,1)$}.
\]
The lemma follows from the fact that $0 < \kappa^{-1}\leq 1$.
\end{comment}

\remark{It is shown in \cite{lu91so} that \eqref{eq:rankH} is a good approximation for the number of iterations needed to get the relative error less than $\delta$, provided that $\kappa \gg 1$. It is shown in \cite{wachspress13ad} that \eqref{eq:rankH} is a good approximation provided that $r \ge 3$. Here, we have shown that \eqref{eq:rankH} is sufficient to guarantee a strict bound on the error.}

\remark{The choice of parameters $p_1,\ldots,p_r$ which minimizes $\max_{x \in [a,b]}|\phi(x)|^2$ is given by the formula $p_k = b\text{dn}\left[\tfrac{2k-1}{2r}I(\sqrt{1-\kappa^{-2}}),\sqrt{1-\kappa^{-2}}\right],$ where $\text{dn}[z,\tau]$ is the Jacobi elliptic function. This function is defined as $\text{dn}[z,\tau] = \sqrt{1-\tau^2\sin^2\varphi}$, where $\varphi$ satisfies $\int_{0}^{\varphi}(1-\tau^2\sin^2\theta)^{-1/2}\,d\theta = z$. If the Jacobi elliptic function $\text{dn}$ is not available, a suboptimal choice of parameters $p_1,\ldots,p_r$ is given by $p_k = a^{\tfrac{2k-1}{2r}}b^{\tfrac{2r-2k+1}{2r}}$, i.e., we can pick the parameters to be evenly spaced on a log scale.}

\remark{If the matrix $\mA$ is diagonal, each iteration of the CF-ADI algorithm above will take $O(N)$ operations. Hence, the matrix $\mZ$ can be computed in $O(rN)$ operations.}

\frenchspacing
\bibliographystyle{unsrt}
\bibliography{bibfileFAST}

\end{document}